\documentclass[10pt,reqno]{amsart}
\usepackage{amsmath, amsthm, amssymb, enumerate}
\usepackage[usenames]{color}
\usepackage{eepic,epic}
\usepackage{esint}
\newtheorem{thm}{Theorem}[section]

\newtheorem{lem}[thm]{Lemma}
\newtheorem{prop}[thm]{Proposition}
\theoremstyle{definition}

\theoremstyle{remark}

\numberwithin{equation}{section}

\newcommand{\R}{\mathbb{R}}
\newcommand{\ep}{\varepsilon}

\begin{document}

\title[Semilinear nonlocal elliptic equations under noncompact settings]
{Infinitely many solutions for semilinear nonlocal elliptic equations under noncompact settings}

\author{Woocheol Choi}
\address[Woocheol Choi]{Department of Mathematical Sciences, Seoul National University, 1 Gwanakro, Gwanak-gu, Seoul 151-747, Republic of Korea}
\email{chwc1987@math.snu.ac.kr}

\author{Jinmyoung Seok}
\address[Jinmyoung Seok]{Department of Mathematics, Kyonggi University,
154-42 Gwanggyosan-ro, Yeongtong-gu, Suwon 443-760, Republic of Korea}
\email{jmseok@kgu.ac.kr}

\begin{abstract}
In this paper, we study a class of semilinear nonlocal elliptic equations posed on settings without compact Sobolev embedding.
More precisely, we prove the existence of infinitely many solutions to
the fractional Brezis-Nirenberg problems on bounded domain.
\end{abstract}


\maketitle

\section{Introduction}
The aim of this paper is to prove the existence of infinitely many solutions to some kinds of semilinear elliptic equations involving the fractional Laplace operator $(-\Delta)^s$
which is nonlocal in nature. The fractional Laplace operator arises when we consider the infinitesimal generator of the L\'{e}vy stable diffusion process in probability theory or
the fractional quantum mechanics for particles on stochastic fields. For further motivations and backgrounds, we refer to \cite{FQT} and references therein.
Recently, the semilinear nonlocal elliptic equations, which are denoted by
\begin{equation}\label{frac-semi-eqns}
(-\Delta)^{s}u = f(x,u)\quad \text{in } \Omega \subset \R^N, \quad 0 < s < 1,
\end{equation}
have been widely studied under various contexts.
In this paper, we are interested in equations of the form \eqref{frac-semi-eqns}, which are posed on function spaces without compact Sobolev embedding.
We shall study the fractional Brezis-Nirenberg problems on bounded domains.

We first introduce a fractional Brezis-Nirenberg problem.
Let $\Omega$ be a smooth bounded domain in $\R^N$. For given $s \in (0, 1)$ and $\mu > 0$, the following problem
\begin{equation}\label{eq-main}
\left\{ \begin{array}{ll} (-\Delta)^s u = |u|^{\frac{2N}{N-2s}-2}u +\mu u &\quad \textrm{in}~\Omega,
\\
u=0 &\quad \textrm{on}~\partial \Omega
\end{array}
\right.
\end{equation}
is called the fractional Brezis-Nirenberg problem.
As in \cite{CT}, the fractional Laplacian $(-\Delta)^{s}$ is defined through the spectral decomposition of the usual Laplacian with zero Dirichlet condition.
The precise definition is given in Section 2.
The eqution \eqref{eq-main} is a fractional version of the classical Brezis-Nirenberg problem,
\begin{equation}\label{eq-BN}
\left\{ \begin{array}{ll} -\Delta u = |u|^{\frac{2N}{N-2}-2}u + \mu u &\quad \textrm{in}~\Omega,
\\
u=0 &\quad \textrm{on}~\Omega.
\end{array}
\right.
\end{equation}

Due to the loss of compactness of Sobolev embedding $H_0^1 (\Omega) \hookrightarrow L^{\frac{2N}{N-2}}(\Omega)$
and $H_0^s (\Omega) \hookrightarrow L^{\frac{2N}{N-2s}}(\Omega)$,
more careful analysis is required to construct nontrivial solutions to the equations \eqref{eq-main} and \eqref{eq-BN} than equations with sub-critical nonlinearities.
In a celebrated paper \cite{BN}, Brezis and Nirenberg first studied the existence of a positive solution to \eqref{eq-BN}.
Let $\lambda_1$ and $\phi_1$ respectively denote the first eigenvalue of $-\Delta$ with zero Dirichlet boundary condition on $\Omega$ and a corresponding positive eigenfunction.
By testing $\phi_1$ to \eqref{eq-BN}, it is easy to see that if $\mu \geq \lambda_1$, there is no positive solution to \eqref{eq-BN}.
Also, the well-known Pohozaev's identity says that if $\mu \leq 0$ and $\Omega$ is star-shape, there is no nontrivial solutions to \eqref{eq-BN}.
Thus, one can deduce that the condition $\mu \in (0, \lambda_1)$ is necessary for \eqref{eq-BN} to admit a positive solution for general smooth domains $\Omega$.
Brezis and Nirenberg proved in \cite{BN} that if $N \geq 4$, the above condition is sufficient. In other words, there is a positive least energy solution to \eqref{eq-BN}
for all $\mu \in (0, \lambda_1)$.

Since the work of Brezis and Nirenberg, many research papers  have been devoted to study the problem \eqref{eq-BN}.
One of most important works is made by Devillanova and Solimini who proved in \cite{DS} the existence of infinitely many solutions for the problem \eqref{eq-BN} when $N \geq 7$ and $\mu > 0$.
This work was extended to an analogous problem involving $p$-Laplacian for $1<p<\infty$ by Cao-Peng-Yan \cite{CPY}.
They proved that if $N > p^2 + p$, the following problem
\[
-\Delta_p u = |u|^{p^*-2}u +\mu|u|^{p-2}u \quad \text{in } \Omega, \qquad u = 0 \quad \text{on } \partial\Omega,
\]
where $\mu > 0$ and $p^* = \frac{pN}{N-p}$, has infinitely many nontrivial solutions.

The equation \eqref{eq-main} was first studied by Tan \cite{T} for $s=\frac{1}{2}$, where he obtained existence of a positive solution. It was extended to the problem with nonlinearity $f(u) = u^{\frac{N+2s}{N-2s}} + \mu u^{q}$ for $s < \min\{N/2, 1\}$, $\mu \in \mathbb{R}$ and $q \in (0, \frac{N+2s}{N-2s})$ in the work of Barrios-Coloado-Pablo-S\'anchez \cite{BCPS2}.
Choi-Kim-Lee  \cite{CKL} investigated the asymptotic behavior of solutions to \eqref{eq-main} as $\mu$ goes to zero. The equation \eqref{eq-main} is also related to the geometric problem called \emph{the Fractional Yamabe problem}. Concerning this problem we refer to Chang-Gonz\'alez \cite{CG}, Gonz\'alez-Qing \cite{GQ}, and Gonz\'alez-Wang \cite{GW}. 
On the other hand, it is worth to mention that 
a fractional Laplacian with exterior zero condition can be also defined 
as an integral operator. In this setting, Servadei and Valdinoci obtained the existence of solutions for the sub-critical problem \cite{SV2} and Brezis-Nireberg problem \cite{SV1, SV3}. In addition, Servadi \cite{Se} showed that the sub-critical problem possesses infinitely many solutions. 
%Sub-critical problems with the fractional Laplacian has been studied in \cite{BCPS1, CT, T2} on the bounded domains. The fractional Yamabe problem was studied by Gonzalez-Qing \cite{GQ}.
However, to the best of our knowledge, there has been no result in literature, which deals with the existence of infinitely many solutions to the equation \eqref{eq-main} having the critical exponent. 
Compared to the sub-critical problems, the main difficuly lies in the fact that one can not use the standard variational technique to obtain a nontrivial solution because the Palais-Smale condition fails to hold due to the loss of compact Sobolev embedding. 
We need to overcome this difficulty to obtain our main result, which extends the Devillanova and Solimini's result in \cite{DS} to the fractional case.
\begin{thm}\label{thm-main-1}
Let $s \in (0, 1)$ and $\mu > 0$ be given. Suppose $N > 6s$. Then the equation \eqref{eq-main} admits infinitely many nontrivial solutions.
\end{thm}
\noindent We shall prove Theorem \ref{thm-main-1} by following Devillanova and Solimini's ideas in \cite{DS}.
The main strategy in these ideas is to consider approximating subcritical problems for which one can show that there are infinitely many nontrivial solutions.
In other words, we consider subcritical problems
\begin{equation}\label{eq-subcritical}
\left\{\begin{array}{ll} (-\Delta)^s u = |u|^{\frac{2N}{N-2s} -2 -\ep} u + \mu u &\quad \textrm{in}~\Omega,
\\
u = 0&\quad \textrm{on}~\partial \Omega,
\end{array}\right.
\end{equation}
for small $\ep > 0$.
From the sub-criticality of the problems, one can verify by using standard variational methods that for every small $\ep > 0$,
\eqref{eq-subcritical} admits infinitely many nontrivial solutions in a fractional Sobolev space $H^s_0(\Omega)$. (We will define $H^s_0(\Omega)$ precisely in Section 2.)
In this regard we shall prove the following compactness result to obtain nontrivial solutions to our original equation \eqref{eq-main}.
\begin{thm}\label{thm-uniform-bound}
Assume $N > 6s$. Let $\{u_n\}$ be a sequence of solutions to \eqref{eq-subcritical} with $\ep =\ep_n \rightarrow 0$ as $n \to \infty$
and $\sup_{n \in \mathbb{N}} \| u_n \|_{H^s_0(\Omega)} < \infty$. Then $\{u_n \}$ converges strongly in $H^s_0(\Omega)$ up to a subsequence.
\end{thm}
\noindent Combining Theorem \ref{thm-uniform-bound} with a well-known topological genus theory, we will see in Section 6 that there are infinitely many nontrivial solutions to \eqref{eq-main}. The proof of Theorem \ref{thm-uniform-bound} will be the main task of this paper, which requires a series of new delicate analysis.
\

It turns out from several technical reasons that studying our nonlocal equations \eqref{eq-main} and \eqref{eq-subcritical} directly is not suitable for establishing Theorem \ref{thm-uniform-bound}, and it is advantageous to consider so-called $s$-harmonic extension problems \eqref{u0inc} and \eqref{eq-main-q}, which are equivalent to \eqref{eq-main} and \eqref{eq-subcritical} respectively.
As we will see in Section 2, the equations \eqref{u0inc} and \eqref{eq-main-q} are local so that they are much easier to deal with than nonlocal ones,
but the domain of problems are changed from $\Omega$ to the half-infinite cylinder $\mathcal{C}:= \Omega \times[0,\infty)$.
This kind of localization was initiated by Caffarelli-Sylvestre \cite{CS} in which the domain under consideration is the whole space $\R^N$, and has been made for bounded domains by many authors \cite{BCPS1, CT, T2}.

By virtue of considering localized equations, one can easily obtain the concentration compactness principle of Struwe \cite{S} for a sequence of solutions to a local equation \eqref{eq-main-q}.
This principle says that a bounded sequence of solutions to \eqref{eq-main-q} in a Sobolev space consists of a function that the sequence weakly converges,
finitely many bubbles that may possibly exist and a function that strongly converges to zero (see Lemma \ref{lem-cc-bounded}). Under this decomposition, to get the compactness, we need to get rid of possibility that bubbles appear.
This will be achieved by arguing indirectly, i.e., we assume there exist bubbles in the sequence and get a contradiction.
For this, an important issue is to verify a sharp bound of the solutions on some thin annuli near a bubbling point. We devote a large part of this paper to obtain it.
We give a full detail of ideas for the proof for Theorem \ref{thm-uniform-bound} in Section 3.
After the proof of Theorem \ref{thm-uniform-bound}, we shall complete the proof of Theorem \ref{thm-main-1} by using a min-max principle combined with the topological genus.

The rest of the paper is organized as follows. In Section 2, we review the fractional Laplacian, $s$-harmonic extension and the extended local problems posed on half-infinite cylinders.
We also arrange some basic lemmas which will be used throughout the paper. In Section 3, we give basic settings and ideas for the proof for Theorem \ref{thm-uniform-bound}.
By following these ideas, we complete the proof of Theorem \ref{thm-main-1} and Theorem \ref{thm-uniform-bound} in subsequent sections 4, 5 and 6.
In Appendix A we prove a technical lemma which will be essentially used in Section 5.
In Appendix B, we prove a lemma which corresponds a non-local version of Moser's iteration method.
Finally in Appendix C, we establish so-called local Pohozaev identity for solutions to \eqref{eq-main-q}, that is a main ingredient for obtaining compactness of a sequence of solutions to \eqref{eq-main-q}.

\bigskip
\noindent \textbf{Notations.}

\medskip
\noindent Here we list some notations which will be used throughout the paper.

\noindent - We shall denote by $2^{*}(s)$ the critical exponent $\frac{2N}{N-2s}$.

\noindent - The letter $z$ represents a variable in the $\mathbb{R}^{n+1}$. Also, it is written as $z = (x,t)$ with $x \in \mathbb{R}^n$ and $t \in \mathbb{R}$.

\noindent - For a domain $D \subset \mathbb{R}^d$ with $d=N$ or $d=N+1$, the map $\nu = (\nu_1, \cdots, \nu_n): \partial D \to \mathbb{R}^d$ denotes the outward pointing unit normal vector on $\partial D$.

\noindent - $dS$ stands for the surface measure. Also, a subscript attached to $dS$ (such as $dS_x$ or $dS_z$) denotes the variable of the surface.

\noindent - $C > 0$ is a generic constant that may vary from line to line.

\

\section{Mathematical frameworks and preliminaries}
\subsection{Fractional Sobolev spaces, fractional Laplacians and $s$-harmonic extensions}\label{subsec_frac_Sob}
Let $\Omega$ be a smooth bounded domain in $\mathbb{R}^N$
and $\{0<  \lambda_k, \phi_k\}_{k=1}^{\infty}$ be the complete system of eigenvalues and eigenfunctions of eigenvalue problems:
\[\left\{ \begin{array}{ll}
- \Delta \phi_k = \lambda_k \phi_k &\text{in}~ \Omega,\\
\phi_k = 0 &\text{on}~ \partial \Omega,\\
\end{array}\right.\]
such that $\|\phi_k\|_{L^2(\Omega)} = 1$ and $\lambda_1 < \lambda_2 \le \lambda_3 \le \cdots$.
By following the paper \cite{CT},
we define a fractional Sobolev space $H_0^s (\Omega)$ for $s \in (0, 1)$ by
\begin{equation}\label{H_0^s}
H_0^s (\Omega) = \left\{ u = \sum_{k=1}^{\infty} a_k \phi_k \in L^2 (\Omega) : \sum_{k=1}^{\infty} \lambda_k^{s}a_k^2 < \infty \right\},
\end{equation}
which is a Hilbert space equipped with an inner product:
\[
\langle u, v \rangle_{H_0^s(\Omega)} = \sum_{k=1}^{\infty} \lambda_k^sa_k b_k
\qquad \text{for} \quad
u = \sum_{k=1}^{\infty} a_k \phi_k,\ v = \sum_{k=1}^{\infty} b_k \phi_k \in H_0^s(\Omega).
\]
The fractional Laplace operator with zero Dirichlet condition,
$(-\Delta)^{s} : H_0^s(\Omega) \to H_0^{-s}(\Omega)$
 is defined by
\[
\langle (-\Delta)^{s}u,v \rangle_{H_0^{-s}(\Omega)} = \langle u, v \rangle_{H_0^s(\Omega)}
\quad \text{ for all } v \in H^s_0(\Omega),
\]
where $H_0^{-s}(\Omega)$ denotes the dual space of $H_0^s(\Omega)$.
Observe that for any function $ u = \sum_{k=1}^{\infty} a_k \phi_k \in H_0^{2s}(\Omega)$, 
$(-\Delta)^{s}u$ has a unique realization in $L^2(\Omega)$ such that
\[(-\Delta)^{s}u = \sum_{k=1}^{\infty} a_k \lambda_k^s \phi_k.\]
Thus we see the inner product is written by 
\[\left\langle u, v \right\rangle_{H_0^s(\Omega)} = \int_{\Omega} (-\Delta)^{s/2}u \cdot (-\Delta)^{s/2}v \,dx \quad \text{for } u, v \in H_0^s(\Omega)\]
and if $u \in H^{2s}_0(\Omega)$, an integration by parts formula holds as follows:
\[
\int_{\Omega} (-\Delta)^{s/2}u \cdot (-\Delta)^{s/2}v\,dx = \int_{\Omega} (-\Delta)^s u\cdot v\,dx.
\]

Next, we consider the whole space $\mathbb{R}^N$.
For $s \in (0, 1)$, we define a function space $D^s(\mathbb{R}^N)$ by
\[D^s(\mathbb{R}^N) = \left\{u \in L^{2^*(s)}(\mathbb{R}^N): \|u\|_{D^s(\mathbb{R}^N)} := \left(\int_{\mathbb{R}^N} |\xi|^{2s}|\hat{u}(\xi)|^2 d\xi\right)^{1 \over 2} < \infty \right\} \]
where $\hat{u}$ denotes the Fourier transform of $u$.
We call $D^s(\mathbb{R}^N)$ the homogeneous fractional Sobolev space.
Note that $D^s(\R^N)$ is a Hilbert space equipped with an inner product
\[
\left\langle u, v \right\rangle_{D^s(\R^N)} = \int_{\R^N}|\xi|^{2s}\hat u(\xi)\hat v(\xi)\,d\xi.
\]
We also define the fractional Laplace operator on $\R^N$, 
$(-\Delta)^{s}: D^{s}(\mathbb{R}^N) \to D^{-s}(\mathbb{R}^N)$ by
\[
\langle (-\Delta)^{s}u, v\rangle_{D^{-s}(\R^N)} = \left\langle u, v \right\rangle_{D^s(\R^N)}
\text{ for all } v \in D^s(\R^N),,
\]
where $D^{-s}(\R^N)$ is the dual of $D^s(\R^N)$.
Then, one can easily check that if $u \in D^{2s}(\R^N)$, we have $(-\Delta)^s u \in L^2(\R^N)$ such that
\[ (-\Delta)^{s}u = \frak{F}^{-1}[|\xi|^{2s} \hat{u}(\xi)]\]
where $\frak{F}^{-1}$ denotes the inverse Fourier transform.
We see for $u, v \in D^s(\R^N)$
\[
\left\langle u, v \right\rangle_{D^s(\R^N)} = \int_{\R^N} (-\Delta)^{s/2}u \cdot (-\Delta)^{s/2}v
\]
and if $u \in D^{2s}(\R^N),\, v \in D^{s}(\R^N)$, we can integrate by parts:
\[
\int_{\R^N} (-\Delta)^{s/2}u \cdot (-\Delta)^{s/2}v = \int_{\R^N} (-\Delta)^s u\cdot v.
\]
Finally, the notation $H^s(\R^N)$ denotes the standard fractional Sobolev space defined as
\[
H^s(\R^N) = D^s(\R^N) \cap L^2(\R^N).
\]

Now we introduce the concept of $s$-harmonic extension of a function $u$ on $\Omega$. Here $\Omega$ is either a whole space $\mathbb{R}^N$ or a smooth bounded domain.
This provides a way to represent fractional Laplace operators as a form of Dirichlet-to-Neumann map.
To do this, we need to define additional function spaces on the half infinite cylinder $\mathcal{C} = \Omega \times (0, \infty)$.
Let $L^2(t^{1-2s}, \mathcal{C})$ denote a weighted Lebesgue space defined by the set of all measurable functions $U : \mathcal{C} \to \R$ satisfying
\[
\| U \|_{L^2(t^{1-2s}, \mathcal{C})} :=
\left(\int_{\mathcal C}t^{1-2s}U^2\,dxdt\right)^{\frac12} < \infty.
\]
A weighted Sobolev space $H^1(t^{1-2s}, \mathcal{C})$ is defined by
\[
H^1(t^{1-2s}, \mathcal{C}) = \{U \in L^2(t^{1-2s}, \mathcal{C}) : \nabla U \in L^2(t^{1-2s}, \mathcal{C})\}.
\]
Then it is a Hilbert space equipped with an inner product
\[
\left\langle U, V \right\rangle_{H^1(t^{1-2s}, \mathcal{C})} = \int_{\mathcal C} t^{1-2s}(\nabla U \cdot \nabla V + UV)\,dxdt.
\]

Suppose that $\Omega$ is smooth and bounded. We set the lateral boundary $\partial_L\mathcal{C}$ of $\mathcal{C}$ by
\[
\partial_L\mathcal{C} := \partial \Omega \times [0,\infty).
\]
Then the function space $H^1_{0}(t^{1-2s}, \mathcal{C})$ defined by the completion of
\[ C_{0,L}^{\infty}(\mathcal{C})
:= \left\{ U \in C^{\infty}\left(\overline{\mathcal{C}}\right) : U = 0 \text{ on } \partial_L\mathcal{C} \right\}\]
with respect to the norm
\begin{equation}\label{weighted_norm}
\|U\|_{H^1_{0}(t^{1-2s}, \mathcal{C})} = \left( \int_{\mathcal{C}} t^{1-2s} |\nabla U|^2\,dxdt \right)^{1/2},
\end{equation}
is also a Hilbert space endowed with an inner product
\[
(U,V)_{H^1_{0}(t^{1-2s}, \mathcal{C})} = \int_{\mathcal{C}} t^{1-2s} \nabla U \cdot \nabla V \,dxdt.
\]
It is verified in \cite[Proposition 2.1]{CS} and \cite[Section 2]{T2} that $H^s_0(\Omega)$ is the continuous trace of $H^1_{0}(t^{1-2s}, \mathcal{C})$, i.e.,
\begin{equation}\label{eq_Sobo_trace}
H_0^s(\Omega) = \{u = \text{tr}|_{\Omega \times \{0\}}U: U \in H^1_{0}(t^{1-2s}, \mathcal{C})\},
\end{equation}
and
\begin{equation}\label{ineq-trace1}
\|U(\cdot, 0)\|_{H^s_0(\Omega)} \leq C\|U\|_{H^1_0(t^{1-2s},\mathcal{C})}
\end{equation}
for some $C > 0$, independent of $U \in H^1_0(t^{1-2s}, \mathcal{C})$.

When $\Omega = \R^{N}$(in this case $\mathcal{C} = \R^{N+1}_+$), one can define a weighted homogeneous Sobolev space $D^1(t^{1-2s}, \R^{N+1}_+)$ as the completion of
$C_c^{\infty}\left(\overline{\mathbb{R}^{N+1}_+}\right)$ with respect to the norm
\[
\|U\|_{D^1(t^{1-2s}, \R^{N+1}_+)} := \left(\int_{\R^{N+1}_+}t^{1-2s}|\nabla U|^2\,dxdt\right)^{1/2}.
\]
Similarly, it holds by taking trace that
\[
D^s(\R^N) = \{u = \text{tr}|_{\R^N \times \{0\}}U : U \in D^1(t^{1-2s}, \R^{N+1}_+)\}
\]
and
\begin{equation}\label{ineq-trace2}
\|U(\cdot, 0)\|_{D^s(\mathbb{R}^N)} \le C \|U\|_{D^1(t^{1-2s}, \R^{N+1}_+)}
\end{equation}
for some $C > 0$ independent of $U \in D^1(t^{1-2s}, \R^{N+1}_+)$.

Now, we are ready to introduce $s$-harmonic extensions of $u \in H^s_0(\Omega)$ for bounded $\Omega$ or $u \in D^s(\R^N)$,
that can be thought as the inverses of the trace processes above.
Let $u \in H^s_0(\Omega)$ and  $v \in D^s(\R^N)$.
By works of Caffarelli-Silvestre \cite{CS} (for $\mathbb{R}^N$), Cabr\'e-Tan \cite{CT} (for bounded domains $\Omega$, see also \cite{ST, BCPS1, T2}),
it is known that there are unique functions $U \in H^1_0(t^{1-2s},\mathcal{C})$ and $V \in D^1(t^{1-2s},\R^N)$ which satisfies the equation
\begin{equation}\label{s-extension}
\left\{ \begin{array}{ll} \text{div}(t^{1-2s} \nabla U) = 0 &~\text{in}~ \mathcal{C},\\
U = 0 & ~\text{on}~ \partial_L \mathcal{C},\\
U(x,0)= u(x) &~ \text{for}~ x \in \Omega,
\end{array}\right.
\end{equation}
and
\begin{equation}\label{s-extension-2}
\left\{ \begin{array}{ll} \text{div}(t^{1-2s} \nabla V) = 0 &~\text{in}~ \R^{N+1}_+,\\
V(x,0)= v(x) &~ \text{for}~ x \in \R^N
\end{array}\right.
\end{equation}
respectively in distributional sense.
Moreover, if $u$ and $v$ are compactly supported and smooth, then the following limits
\[
\partial_{\nu}^{s}W(x,0):= -C_s^{-1} \left(\lim_{t \rightarrow 0+} t^{1-2s} \frac{\partial W}{\partial t}(x,t)\right) \quad  \text{with } C_s := \frac{2^{1-2s}\Gamma(1-s)}{\Gamma(s)},
\quad W = U \text{ or } V,
\]
are well defined and one must have
\begin{equation}\label{frac-localize}
(-\Delta)^{s}w = \partial_{\nu}^s W(x,0), \quad w = u \text{ or } v.
\end{equation}
We call these $U$ and $V$ the $s$-harmonic extensions of $u$ and $v$.
We point out that by a density argument, the relation \eqref{frac-localize} is satisfied in weak sense for $u \in H^s_0(\Omega)$ and $v \in D^s(\R^N)$.
In other words, it holds that for every $u$ and $\phi  \in H^s_0(\Omega)$,
\[
\left\langle u, \phi \right\rangle_{H^s_0(\Omega)} = C_s^{-1}\left\langle U, \Phi \right\rangle_{H^1_0(t^{1-2s},\mathcal{C})}
\quad \text{where} \quad U,\, \Phi = \text{ $s$-harmonic extensions of } u,\, \phi
\]
and the analogous statement holds for every $v$ and $\phi \in D^s(\R^N)$.
Thus the trace inequalities \eqref{ineq-trace1} and \eqref{ineq-trace2} are improved as
\[
\|U(\cdot, 0)\|_{H^s_0(\Omega)} = C_s^{-1}\|U\|_{H^1_0(t^{1-2s},\mathcal{C})},\quad \|U(\cdot, 0)\|_{D^s(\mathbb{R}^N)} = C_s^{-1}\|U\|_{D^1(t^{1-2s}, \R^{N+1}_+)}
\]
if \eqref{s-extension} and \eqref{s-extension-2} hold repectively.

By the above discussion, one can deduce that a function $u \in H^s_0 (\Omega)$ is a weak solution to  the nonlocal problem \eqref{eq-main}
if and only if its $s$-harmonic extension $U \in H^1_0(t^{1-2s},\mathcal{C})$ is a weak solution to the local problem
\begin{equation}\label{u0inc}
\left\{ \begin{array}{ll} \text{div}(t^{1-2s} \nabla U )= 0&\quad \text{in}~ \mathcal{C},\\
U = 0 &\quad \text{on} ~\partial_L \mathcal{C},
\\
\partial_{\nu}^{s} { U} = |U|^{2^* (s)-2} U(x,0) + \mu U(x,0) & \quad \text{on} ~\Omega \times \{0\},
\end{array}
\right.
\end{equation}
and similarly the problem $\eqref{eq-subcritical}$ corresponds to
\begin{equation}\label{eq-main-q}
\left\{ \begin{array}{ll} \textrm{div}(t^{1-2s}\nabla U) = 0 &\quad \textrm{in}~\mathcal{C},
\\
U=0 &\quad \textrm{on}~\partial_L \mathcal{C},
\\
\partial_{\nu}^s U = |U|^{p} U(x,0) + \mu U(x,0) &\quad \textrm{on}~\Omega \times \{0\},
\end{array}
\right.
\end{equation}
where $1 < p < 2^* (s)- 2$.
By weak solutions, we mean the following:
Let $g \in L^{\frac{2N}{N+2s}}(\Omega)$. Given the problem
\begin{equation}\label{eq-weak1}
\left\{\begin{array}{ll}
(-\Delta)^{s}u = g(x) &\quad \textrm{in}~\Omega,
\\
u = 0 &\quad \textrm{on}~\partial\Omega,
\end{array}\right.
\end{equation}
we say that a function $u \in H_0^{s}(\Omega)$ is a weak solution of \eqref{eq-weak1} provided
\begin{equation}
\int_{\Omega} (-\Delta)^{s/2}u \cdot (-\Delta)^{s/2}\phi \,dx = \int_{\Omega} g(x)\phi(x)\,dx
\end{equation}
for all $\phi \in H^s_0(\Omega)$.
Also, given the problem
\begin{equation}\label{eq-weak2}
\left\{
\begin{array}{ll}
\textrm{div}(t^{1-2s} \nabla U) = 0 &\quad \textrm{in}~\mathcal{C},
\\
U = 0 &\quad \textrm{on}~\partial_{L}\mathcal{C},
\\
\partial_{\nu}^{s} U = g(x)&\quad \textrm{on}~\Omega \times\{0\},
\end{array}
\right.
\end{equation}
we say that a function $U \in H_0^1(t^{1-2s}, \mathcal{C})$ is a weak solution of \eqref{eq-weak2} provided
\begin{equation}
\int_{\mathcal{C}} t^{1-2s}  \nabla U(x,t) \cdot \nabla \Phi(x,t)\, dx dt = C_s\int_{\Omega} g(x)\Phi(x,0)\,dx
\end{equation}
for all $\Phi \in H_0^1(t^{1-2s}, \mathcal{C})$.

\subsection{Weighted Sobolev and Sobolev-trace inequalities}\label{subsec_sobolev_trace}
Given any $\lambda > 0$ and $\xi \in \mathbb{R}^N$, let
\begin{equation}\label{bubble}
w_{\lambda, \xi} (x) = \mathfrak{c}_{N,s} \left( \frac{\lambda}{\lambda^2+|x-\xi|^2}\right)^{\frac{N-2s}{2}} \quad \text{for } x \in \mathbb{R}^N,
\end{equation}
where
\begin{equation}\label{cns}
\mathfrak{c}_{N,s} = 2^{\frac{N-2s}{2}} \left( \frac{\Gamma \left( \frac{N+2s}{2}\right)}{\Gamma \left( \frac{N-2s}{2}\right)}\right)^{\frac{N-2s}{4s}}.
\end{equation}
Then we have the following Sobolev inequality
\[
\left( \int_{\mathbb{R}^N} |u|^{2^*(s)} dx \right)^{\frac{1}{2^*(s)}} \leq \mathcal{S}_{N,s} \left( \int_{\mathbb{R}^N} |(-\Delta)^{s/2} u|^2 dx \right)^{1 \over 2}, \quad u \in H^s_0(\Omega),
\]
which attains the equality exactly when $u(x) = cw_{\lambda,\xi}(x)$ for any $c > 0,\, \lambda>0$ and $\xi \in \mathbb{R}^N$ (we refer to \cite{L2, CL, FL}).
Here,
\begin{equation}\label{ans}
\mathcal{S}_{N,s} = 2^{-2s} \pi^{-s} \frac{\Gamma \left(\frac{N-2s}{2}\right)}{\Gamma \left( \frac{N+2s}{2}\right)} \left[ \frac{\Gamma(N)}{\Gamma(N/2)}\right]^{2s/N}.
\end{equation}
It follows that for the Sobolev trace inequality
\begin{equation}\label{eq-sharp-trace-2}
\left( \int_{\mathbb{R}^N} |U(x,0)|^{2^*(s)} dx \right)^{\frac{1}{2^*(s)}} \leq \frac{\mathcal{S}_{N,s}}{\sqrt{C_s}} \left( \int_{\mathbb{R}^{N+1}_{+}} t^{1-2s} |\nabla U(x,t)|^2 dx dt \right)^{1 \over 2},
\quad U \in D^1(t^{1-2s}, \R^{N+1}_+),
\end{equation}
the equality is attained exactly by $U(x,t) = c W_{\lambda,\xi}(x,t)$, where $W_{\lambda,\xi}(x,t)$ is the $s$-harmonic extension of $w_{\lambda,\xi}$.
By zero extension, we also have
\begin{equation}\label{eq-sharp-trace}
\left( \int_{\Omega} |U(x,0)|^{2^*(s)} dx \right)^{\frac{1}{2^*(s)}} \leq \frac{\mathcal{S}_{N,s}}{\sqrt{C_s}} \left( \int_{\mathcal{C}} t^{1-2s} |\nabla U(x,t)|^2 dx dt \right)^{1 \over 2},
\quad U \in H^1_0(t^{1-2s}, \mathcal{C}).
\end{equation}
As an application, we obtain the following estimate.
\begin{lem}\label{lem-trace}
Let $w \in L^{p}(\Omega)$  for some $p < \frac{N}{2s}$. Assume that $U$ is a weak solution of the problem
\begin{equation}\label{eq-lem-basic}
\left\{\begin{array}{ll}
\textrm{\em div}(t^{1-2s} \nabla U) = 0&\quad \textrm{in}~\mathcal{C},
\\
U=0&\quad \textrm{on}~\partial_{L}\mathcal{C},
\\
\partial_{\nu}^s U = w &\quad \textrm{on}~\Omega \times \{0\}.
\end{array}
\right.
\end{equation}
Then we have
\begin{equation}
\left\| U(\cdot, 0)\right\|_{L^q (\Omega)} \leq C_{p,q} \left\| w \right\|_{L^p (\Omega)},
\end{equation}
for any $q$ such that $\frac{N}{q} \leq \frac{N}{p} -2s$.
\end{lem}
\begin{proof}
We multiply \eqref{eq-lem-basic} by $|U|^{\beta-1}U$ for some  $\beta >1$ to get
\begin{equation}
 \int_{\Omega} w(x) |U|^{\beta-1} U (x,0)\, dx= \beta \int_{\mathcal{C}} t^{1-2s} |U|^{\beta-1}|\nabla U|^2\, dxdt.
\end{equation}
Then, applying the trace embedding \eqref{eq-sharp-trace} and H\"older's inequality we can observe
\begin{equation}\label{eq-lem-beta}
\left\| |U|^{\frac{\beta+1}{2}}(\cdot,0) \right\|^2_{L^{\frac{2N}{N-2s}}(\Omega)} \leq C_\beta \left\| |U|^{\beta}(\cdot,0)\right\|_{L^{\frac{\beta+1}{2\beta}\cdot \frac{2N}{N-2s}}} \left\| w\right\|_{p},
\end{equation}
where $p$ satisfies $\frac{1}{p} + \frac{(N-2s)\beta}{N(\beta+1)} =1$. Let $q= \frac{N(\beta+1)}{N-2s}$, %It holds that $n/q = n/p -2s$, and
then \eqref{eq-lem-beta} gives the desired inequality.
\end{proof}
We will also make use of the following weighted Sobolev inequality.

\begin{prop}\cite[Theorem 1.3]{FKS}
Let $\Omega$ be an open bounded set in $\mathbb{R}^{N+1}$.
Then there exists a constant $C = C(N, s, \Omega) > 0$ such that
\begin{equation}\label{eq-sobolev-weight}
\left( \int_{\Omega} |t|^{1-2s} |U(x,t)|^{\frac{2(N+1)}{N}} dx dt\right)^{\frac{N}{2(N+1)}}
\leq C \left( \int_{\Omega} |t|^{1-2s} |\nabla U(x,t)|^2 dx dt \right)^{1 \over 2}
\end{equation}
holds for any function $U$ whose support is contained in $\Omega$ whenever the right-hand side is well-defined.
\end{prop}

\subsection{Useful lemmas} Here we prepare some lemmas which will be used importantly throughout the paper.
\begin{lem}\label{lem-maximum}
Suppose that $V \in H^1_0(t^{1-2s},\mathcal{C})$ is a weak solution of the following problem
\begin{equation}\label{eq-lem-v}
\left\{ \begin{array}{ll} \textrm{\em div}(t^{1-2s} \nabla V) = 0&\quad \textrm{on}~\mathcal{C},
\\
V(x,t) =0 &\quad \textrm{on}~\partial_L\mathcal{C},
\\
\partial_{\nu}^{s} V(x,0) = g(x) &\quad \textrm{on}~\Omega \times\{0\}
\end{array}
\right.
\end{equation}
for some nonnegative $g$.
Then $V$ is nonnegative everywhere.
\end{lem}
\begin{proof}
Let $V_{-} = \max\{ 0, -V\}$. By testing $V_{-}$, the definition of weak formulation implies
\begin{equation}
-\int_{\mathcal{C}} t^{1-2s}|\nabla V_{-}|^2\, dxdt = C_s\int_{\Omega} g(x) \cdot V_{-}(x,0)\, dx \geq 0
\end{equation}
and thus
\begin{equation*}
\int_{\mathcal{C}}t^{1-2s} |\nabla V_{-}|^2 (x,t) dx dt = 0.
\end{equation*}
It proves that $V_{-} \equiv 0$. The lemma is proved.
\end{proof}
Next we state a variant of the concentration compactness principle \cite{S} for the extended problems.
%Because the proof of it follows almost same lines of \cite{S} for the problem \eqref{eq-BN},
%Let $F_{\lambda}(x) = \frac{1}{p+1} x^{p+1} + \lambda x$ and $F_{0} (x) = \frac{1}{p+1} x^{p+1}$.
%For $(\mu, q) \in \mathbb{R}^2$ we set
%\begin{equation*}
%I_{\mu, q}(U) = \int_{\mathcal{C}} t^{1-2s} |\nabla U(x,t)|^2 dxdt - \frac{1}{q+1} \int_{\Omega} |U|^{q+1} (x,0) dx - \frac{\mu}{2} \int_{\Omega} U^2 (x,0) dx.
%\end{equation*}
\begin{lem}\label{lem-cc-bounded}
For $n \in \mathbb{N}$ let $U_n$ be a solution of \eqref{eq-main-q} with $p = p_n \rightarrow 2^{*}(s)-2$ such that $\|U_n \|_{H^1_0(t^{1-2s}, \mathcal{C})} < C$ for some $C$ independent of $n \in \mathbb{N}$.
Then, for some $k \in \mathbb{N}$, there are $k$-sequences $\{(\lambda_{n}^j, x_n^j )\}_{n=1}^{\infty} \subset \R_+ \times \Omega,\,
1\leq j \leq k$, a function $V^0 \in H^1_0(t^{1-2s}, \mathcal{C})$
and $k$-functions $V^j  \in D^1(t^{1-2s},\mathbb{R}^{N+1}_{+}),\, 1 \leq j \leq k$ satisfying
\begin{itemize}
\item $U_n \rightharpoonup V^0 $ weakly in $H^1_0(t^{1-2s}, \mathcal{C})$;
\item $U_n -\left(V^0 + \sum_{j=1}^{k}\rho_{n}^{j}( V^{j})\right) \to 0
\text{ in } H^1_0(t^{1-2s}, \mathcal{C}) \text{ as } n \to \infty$, where
\[
\rho_n^j(V^j) = (\lambda_n^j)^{\frac{N}{2^*(s)}}V^j(\lambda_n^j(\cdot-x_n^j));
\]
\item $V^0$ is a solution of \eqref{u0inc}, and $V^j$ are non-trivial solutions of
\begin{equation}\label{eq-critical}
\left\{ \begin{array}{ll}
\textrm{\em div}(t^{1-2s} \nabla V) = 0 &\quad \textrm{in}~\mathbb{R}^{N+1}_{+},
\\
\partial_{\nu}^s V = |V|^{2^* (s) -2} V &\quad \textrm {on}~\mathbb{R}^{N} \times \{0\}.
\end{array}
\right.
\end{equation}
\end{itemize}
Moreover, we have
\begin{equation}\label{eq-cc-cond}
\frac{\lambda_n^i}{\lambda_n^j} + \frac{\lambda_n^j}{\lambda_n^i} + \lambda_n^i \lambda_n^j |x_n^i -x_n^j|^2 \rightarrow \infty~\textrm{as}~ n \rightarrow \infty~\textrm{for all}~ i \neq j.
\end{equation}
\end{lem}
\begin{proof}
The proof follows without difficulty by modifying the proof of the concentration compactness result for \eqref{eq-BN}(see \cite{S, S2}), and we omit the details for the sake of simplicity of the paper.
We refer to the paper \cite{A} where S. Almaraz modified the argument in \cite{S} for studying the boundary Yamabe flow.
His setting corresponds to the case $s=1/2$ of the extended problems considered here.
\end{proof}
It is useful to know the decay rate of any entire solutions to \eqref{eq-critical}.
\begin{lem}\label{lem-entire-decay} Suppose that $V \in D^1(t^{1-2s}, \mathbb{R}^{N+1}_{+})$ is a weak solution of \eqref{eq-critical}.
Then there exists a constant $C>0$ such that
\begin{equation*}
|V(x, 0)| \leq \frac{C}{(1 + |x|)^{N-2s}}.
\end{equation*}
\end{lem}
\begin{proof}
We first show that $V$ is a bounded function. For a sake of convenience, we consider a positive function $U \in D^1(t^{1-2s}, \mathbb{R}_{+}^{N+1})$ such that
\begin{equation}\label{eq-U}
\left\{
\begin{array}{ll} \textrm{div}(t^{1-2s} \nabla U)=0&\quad \textrm{in}~\mathbb{R}^{N+1}_{+},
\\
\partial_{\nu}^s U = |V|^{\frac{N+2s}{N-2s}}&\quad \textrm{on}~\mathbb{R}^N \times \{0\}.
\end{array}
\right.
\end{equation}
Then, it is easy to see $|V|\leq U$ by Lemma \ref{lem-maximum} and
\[
\int_{\R^{N+1}_+} t^{1-2s} |\nabla U|^2\, dx dt \leq \int_{\R^{N+1}_+} t^{1-2s} |\nabla V|^2\, dx dt.
\]
For $T>0$ let $U_T = \min \{ U, T\}$. Multiplying \eqref{eq-U} by $U_T^{2\beta} U$ for $\beta >1$ we obtain
\[
\int_{\mathbb{R}^{N}} |V|^{\frac{N+2s}{N-2s}} \cdot U_T^{2\beta}\cdot U (x,0) dx =  \int_{\mathbb{R}^{N+1}_{+}} t^{1-2s} 2\beta|\nabla U_T|^2 U^{2\beta} + t^{1-2s} |\nabla U|^2 U_T^{2\beta} dx dt.
\]
On the other hand, a direct computation shows
\begin{equation}
|\nabla (UU_T^{\beta})|^2 = U_T^{2\beta} |\nabla U|^2 + (2\beta +\beta^2) U_T^{2\beta} |\nabla U_T|^2.
\end{equation}
Thus we deduce
\[
\int_{\mathbb{R}^{N+1}_{+}} t^{1-2s} |\nabla (U U_T^{\beta})|^2 dx dt \leq C \int |V|^{\frac{N+2s}{N-2s}}\cdot U_T^{2\beta} U (x,0) dx,
\]
and consequently, for $K>0$ we have
\[
\begin{split}
\int_{\mathbb{R}^{N+1}_{+}} t^{1-2s} |\nabla (U U_T^{\beta})|^2 &dx dt \leq C
\int_{U \leq K} |V|^{\frac{N+2s}{N-2s}} \cdot U_T^{2\beta} U dx + C \int_{U >K} |V|^{\frac{N+2s}{N-2s}} \cdot U_T^{2\beta} U dx
\\
&\leq K^{2\beta} C + C \left( \int_{U >K} |V|^{\frac{2N}{N-2s}}(x,0) dx \right)^{\frac{2s}{N}} \left( \int_{\mathbb{R}^{N}} |U_T^{\beta} U (x,0)|^{\frac{2N}{N-2s}} dx\right)^{\frac{N-2s}{N}}
\\
&\leq K^{2\beta} C + C \left( \int_{U >K} |V|^{\frac{2N}{N-2s}}(x,0) dx \right)^{\frac{2s}{N}} \left( \int_{\mathbb{R}^{N+1}_{+}} t^{1-2s} |\nabla (U U_T^{\beta})|^2 dx dt\right).
\end{split}
\]
Choosing a sufficiently large $K>0$, we get
\[
\int_{\mathbb{R}^{N+1}_{+}} t^{1-2s} |\nabla (U U_T^{\beta})|^2 dx dt \leq 2 K^{2 \beta} C.
\]
From this, using the Sobolev-trace inequality and letting $T \rightarrow \infty$, we obtain
\[
\int_{\mathbb{R}^{N}} |V|^{2^* (s)(\beta+1)} (x,0) dx \leq \int_{\mathbb{R}^{N}} U^{2^* (s)(\beta+1)} (x,0) dx \leq C.
\]
Here $\beta >1$ can be chosen arbitrary. Now, we use the following kernel expression (see \cite{CS}),
\[U (x,t) = \int_{\R^N} \frac{C_{N,s}}{(|x-y|^2 +t^2)^{\frac{N-2s}{2}}} |V|^{2^{*}(s)-1}(y,0)\,dy\]
and H\"older's inequality to conclude that $U$ is a bounded function. Therefore, $V$ is a bounded function.

Next we consider the following Kelvin transform with $z = (x,t) \in \R^{N+1}_+$,
\begin{equation}\label{eq-kelvin}
W(z) = |z|^{-{(N-2s)}} V\left( \frac{z}{|z|^2}\right).
\end{equation}
From a direct computation, we see that the function $W$ satisfies
\begin{equation*}
\left\{ \begin{array}{ll}
\textrm{div}(t^{1-2s} \nabla W) = 0 &\quad \textrm{in}~\mathbb{R}^{N+1}_{+},
\\
\partial_{\nu}^s W = |W|^{\frac{4s}{N-2s}}W &\quad \textrm {on}~\mathbb{R}^{N} \times \{0\},
\end{array}
\right.
\end{equation*}
and $\| W\|_{D^1(t^{1-2s},\mathbb{R}^{N+1}_{+})} \leq C \| V\|_{D^1(t^{1-2s}, \mathbb{R}^{N+1}_{+})} \leq C$.
Then, we may apply the same argument for $V$ to show that the function $W$ is bounded on $\mathbb{R}^{N+1}_{+}$. So, we can deduce from \eqref{eq-kelvin} that
\begin{equation*}
|V(z)| \leq C |z|^{-(N-2s)}.
\end{equation*}
This proves the lemma.
\end{proof}

\section{Settings and Ideas for the proof of Theorem \ref{thm-uniform-bound}}
Here we build basic settings and expain ideas for the proof of Theorem \ref{thm-uniform-bound} for a clear exposition of the paper.
The arguments introduced in this section are originally developed by Devillanova and Solimini in \cite{DS}
and also are inspired by a modified approach in the work of Cao, Peng and Yan in \cite{CPY}.
From now on, we will denote the norm of the weighted Sobolev space $H^1_0(t^{1-2s},\mathcal{C})$ by $\|\cdot\|$ for simplicity.

Let $\{U_n\}_{n \in \mathbb{N}} \subset H^1_0(t^{1-2s},\mathcal{C})$ be a sequence of functions which are solutions of \eqref{eq-main-q} with $p =p_n \rightarrow 2^*(s)-2$
such that $\|U_n\|$ is bounded uniformly for $n \in \mathbb{N}$.
What we want to prove is the compactness of the sequence $\{U_n\}_{n \in \mathbb{N}}$ in $H^1_0(t^{1-2s},\mathcal{C})$. For this aim, we shall derive a contradiction after assuming that $\{U_n\}_{n \in \mathbb{N}}$ is noncompact. Under this assumption, Lemma \ref{lem-cc-bounded} says that for some integer $k \geq 1$,
there exist $k$ sequences $\{(x_n^j,\, \lambda_n^{j})\}_{n \in \mathbb{N}} \subset \Omega \times \R_+$ with $\lim_{n \rightarrow \infty} \lambda_{n}^j = \infty$ such that
\eqref{eq-cc-cond} holds and
\begin{equation}\label{eq-decom}
\left\{ \begin{array}{l}
U_n = V^0 + \sum_{j=1}^{k} \rho_{n}^j (V^j ) + R_n,
\\
\lim_{n \rightarrow \infty} \| R_n \| = 0,
\end{array}\right.
\end{equation}
where $V^0$ is a solution to \eqref{u0inc} and $V^j$ is an entire solution of \eqref{eq-critical} for $1\leq j \leq k$.
By taking a subsequence, we may assume without loss of generality
\begin{equation*}
\lambda_{n}^1 \leq \lambda_{n}^2 \leq \cdots \leq \lambda_{n}^k \qquad \forall n \in \mathbb{N}.
\end{equation*}
We just denote $\lambda_{n}^1$ by $\lambda_n$ and $x_{n}^1$ by $x_n$ throughout the paper. In other words, the point $x_n$ correponds the slowest bubbling point and $\lambda_n$ is the corresponding rate of blowup.

We shall derive a contradiction by making use a local Pohozaev identity \eqref{eq-local} on concentric balls with center $x_n$ and radii comparable to $\lambda_n^{-1/2}$.
To do this, we shall show that average(and weighted average) integrals of $|U|^q$ on appropriate annuli around $x_n$ are uniform bounded for $n$ whenever $q > 1$.
This will also enable us to get a sharp weighted $L^2$ estimates for $\nabla U$ on the annuli. This will be accomplished in Section 4 and 5.

Let us explain more on the procedure for the uniform estimates. First, we introduce  in Section 4 a norm which reflects the effect of bubbles in sequence $\{U_n\}_{n=1}^{\infty}$ and show the uniform boundedness of $\{U_n\}$ with respect to this norm.
Let $q_1$ and $q_2$ be real numbers such that $\frac{N}{N-2s}< q_2 < \frac{2N}{N-2s} < q_1 < \infty$. For given two functions $u_1 \in L^{p_1}(\Omega)$ and $u_2 \in L^{q_2} (\Omega)$,
let $\alpha>0$ and $\lambda>0$ satisfy the inequality
\begin{equation}\label{eq-inequality-system}
\left\{ \begin{array}{ll} \| u_1 \|_{q_1} &\leq \alpha,
\\
 \|u_2 \|_{q_2} & \leq \alpha \lambda^{\frac{N}{2^*(s)} - \frac{N}{q_2}}.
\end{array}
\right.
\end{equation}
Then we define for given $q_1,\, q_2, \lambda$, a norm as follows:
\begin{eqnarray}\label{eq-def}
\| u\|_{\lambda, q_1,q_2} = \inf\{ \alpha>0 : \textrm{there exist $u_1$ and $u_2$ such that $|u| \leq u_1 + u_2$ and \eqref{eq-inequality-system} holds }\},
\end{eqnarray}
and we shall prove that
\begin{equation*}
\sup_{n \in \mathbb{N}} \|U_n(\cdot,0)\|_{\lambda_n, q_1, q_2} < \infty.
\end{equation*}
In section 5, we establish the uniform boundedness of the average integrals of $|U|^q$ for any $q >1$ and a sharp weighted $L^2$ estimate for $\nabla U$ on suitable annuli around $x_n$ with widths comparable to $\lambda_n^{-1/2}$.
We first show by combining the result in Section 4 and some delicate arguments in the work of Cao-Peng-Yan \cite{CPY} with a nonlocal version of a lemma by Kilpenl\"ainen-Mal\'y \cite{KM}
that the desired average bounds are valid for at least relatively small range of $q$.
Then a Moser's iteration type argument(Lemma \ref{lem-harnack}) applies to widen the range of $q$ to arbitrary $q > 1$.

With these estimates at hand, we make a contradiction from a local Pohozaev identity in Section 6, which completes the proof of Theorem \ref{thm-uniform-bound}.
\

\section{A refined norm estimate}
As explained in Section 3, we prove in this section the following result.
\begin{prop}\label{prop-q1-q2} For $n \in \mathbb{N}$ let $U_n$ be a solution of \eqref{eq-main-q} with $p = p_n \rightarrow 2^{*}(s)-2$ such that $\|U_n \| < C$ for some $C$ independent of $n \in \mathbb{N}$, which admits the decomposition \eqref{eq-decom}. Then, for any numbers $q_1$ and $q_2$ such that $ \frac{N}{N-2s}< q_2 < \frac{2N}{N-2s}< q_1 < \infty$, we have
\begin{equation*}
\sup_{n} \| U_n (\cdot,0) \|_{\lambda_n, q_1, q_2} < \infty.
\end{equation*}
\end{prop}
We will prove this result through the three lemmas below, proofs of which heavily rely on Lemma \ref{lem-trace}, \ref{lem-maximum} and \ref{lem-entire-decay}.
Let us take a constant $A>0$ such that $x^{p+1} + \mu x \leq 2 x^{2^*(s) -1} + A$ for all $x \geq 0$ and $1< p< 2^{*}(s)-2$. Now we consider a solution $\{ D_n\}_{n \in \mathbb{N}}$ to the problem
\begin{equation}\label{eq-uv}
\left\{ \begin{array}{ll} \textrm{div}(t^{1-2s} \nabla D_n) = 0&\quad \textrm{in}~\mathcal{C},
\\
D_n = 0&\quad \textrm{on}~\partial_{L}\mathcal{C},
\\
\partial_{\nu}^s D_n = 2 |U_n|^{2^{*}(s)-1} + A& \quad \textrm{on}~\Omega \times \{0\}.
\end{array}
\right.
\end{equation}
Then, by Lemma \ref{lem-maximum}, we see that $D_n$ is positive and $|U_n| \leq D_n$.
%Note that
%\begin{equation*}
%\partial_{\nu}^{s} D_n = |\partial_{\nu}^s U_n| = |U_n|^{p_n} \quad \textrm{on}~ \Omega \times \{0\}.
%\end{equation*}2^* (s)
Moreover, using \eqref{eq-decom} for some $C_1 >0$ we see that for some $C_1 = C_1 (k)$ the following inequality holds;
\begin{equation}\label{eq-d-decom}
\partial_{\nu}^{s} D_n \leq C_1 \left( |V_0|^{2^* (s)-2} + \sum_{j=1}^{k} |\rho_n^j (V_j)|^{2^* (s)-2} + |R_n|^{2^* (s) -2} \right) |U_n| + A \quad \textrm{on}~\Omega \times \{0\}.
\end{equation}

We prepare the first lemma, which will be used to handle the remainder term $R_n$ converging to zero in $H_0^1(t^{1-2s}, \mathcal{C})$.
\begin{lem}\label{lem-harnack-basic}
Let $ a \in L^{\frac{N}{2s}}(\Omega)$ and $v \in L^{\infty}(\Omega)$. Suppose a function $U \in H_0^1(t^{1-2s}, \mathcal{C})$ satisfies
 \begin{equation*}
\left\{ \begin{array}{ll} \textrm{\em div}(t^{1-2s} \nabla U)=0 &\quad \textrm{in}~\mathcal{C},
\\
U=0&\quad \textrm{on}~\partial_L \mathcal{C},
\\
\partial_{\nu}^{s} U = a(x) v &\quad \textrm{on}~\Omega \times \{0\}.
\end{array}
\right.
\end{equation*}
Then, for any $\lambda >0$ and $\frac{N}{N-2s} <q_1 < \frac{2N}{N-2s} < q_2 < \infty$ we have
\begin{equation*}
\| U(\cdot,0)\|_{\lambda, q_1, q_2} \leq C_{q_1,q_2} \| a\|_{\frac{N}{2s}} \| v\|_{\lambda, q_1,q_2}.
\end{equation*}
\end{lem}
\begin{proof}
Choose arbitrary positive two functions $v_1  \in L^{\infty}(\Omega)$ and $v_2  \in L^{\infty} (\Omega)$ such that $|v(x)| \leq v_1 (x) + v_2 (x)$ for all $x \in \Omega$.
Then, there exist functions $U_1 \in H_0^1(t^{1-2s}, \mathcal{C})$ and $U_2  \in H_0^1(t^{1-2s}, \mathcal{C})$ satisfying
\begin{equation*}
\left\{ \begin{array}{ll} \textrm{div}(t^{1-2s} \nabla U_i) = 0 &\quad \textrm{in}~\mathcal{C},
\\
U_i =0 &\quad \textrm{on}~\partial_{L}\mathcal{C},
\\
\partial_{\nu}^{s} U_i = |a(x)| v_i &\quad \textrm{on}~\Omega \times \{0\},
\end{array}
\right. \quad i =1,2.
\end{equation*}
We see from Lemma \ref{lem-maximum}, the maximum principle that $|U| \leq U_1 + U_2$ . For given $\beta >1$, one has
\begin{equation*}
0 = \int_{\mathcal{C}} \textrm{div}(t^{1-2s} \nabla U_i ) U_i^{\beta} dz = \int_{\Omega \times \{0\}}|a(x)|v_i (x) U_i^{\beta}(x,0) dx - \int_{\mathcal{C}} t^{1-2s} \nabla U_i \nabla U_i^{\beta} dz,
\end{equation*}
which gives
\begin{equation*}
\int_{\mathcal{C}} t^{1-2s} |\nabla U_i^{\frac{\beta+1}{2}}|^2 dz = C_{\beta} \int_{\Omega \times \{0\}} a (x) v_i (x) U_i^{\beta} (x,0) dx.
\end{equation*}
Applying the Sobolev-trace inequality \eqref{eq-sharp-trace} and H\"older's inequality, we get
\begin{equation}\label{eq-uav}
\| U_i^{\frac{\beta+1}{2}} (x,0)\|_{L^{\frac{2N}{N-2s}}(\Omega)}^2 \leq C \| a\|_{\frac{N}{2s}} \| v_i \|_{\frac{\beta+1}{2} \frac{2N}{N-2s}} \| U_i^{\beta} (x,0)|\|_{L^{\frac{\beta+1}{2\beta} \frac{2N}{N-2s}}}.
\end{equation}
For each $i \in \{1,2\}$ we take the value of $\beta$ such that $q_i = \frac{\beta+1}{2} \frac{2N}{N-2s}$. Then \eqref{eq-uav} gives that
\begin{equation*}
\| U_i (x,0)\|_{L^{q_i}} \leq C \| a\|_{\frac{N}{2s}} \| v_i \|_{L^{q_i}}\quad \forall i =1,2.
\end{equation*}
This and the definition \eqref{eq-def} of $\| \cdot~\|_{\lambda,q_1,q_2}$ yield
\[
\left\| U (\cdot,0)\right\|_{\lambda, q_1, q_2} \leq C \|a\|_{N/2s} \| v\|_{\lambda,q_1,q_2}.
\]
This proves the lemma.
\end{proof}
In the following lemma, we find a particular pair $(q_1, q_2)$ such that $\|~\|_{\lambda_n,q_1,q_2}$ is uniformly bounded.
\begin{lem}\label{lem-start}
Let $\{U_n\}_{n \in \mathbb{N}}$ be the sequence of solutions described in Proposition \ref{prop-q1-q2} and consider the sequence of functions $\{D_n\}_{n \in \mathbb{N}}$ defined in \eqref{eq-uv}. Then, there exists $q_1 \in \left( \frac{2N}{N-2s}, \infty\right)$ and $q_2 \in \left( \frac{N}{N-2s}, \frac{2N}{N-2s}\right)$, and a constant  $C>0$ such that
\begin{equation*}
\sup_{n \in \mathbb{N}}\| D_n (\cdot,0)\|_{\rho_n,q_1,q_2} \leq C.
\end{equation*}
\end{lem}
\begin{proof}
For $1 \leq i \leq 3$ we consider the functions $D_n^i  \in H^1_0(t^{1-2s},\mathcal{C})$ such that
\begin{equation*}
\left\{\begin{split}
& \textrm{div}(t^{1-2s} \nabla D_i ) =0\quad \textrm{in}~\mathcal{C}, \quad 1 \leq i \leq 3,
\\
&D_i = 0\quad \textrm{on}~\partial_L \mathcal{C}, \quad 1 \leq i \leq 3,
\\
&\partial_{\nu}^s D_n^1 = C_1 (|V_0|^{2^* (s) -2}) |U_n| + A,
\\
&\partial_{\nu}^{s} D_n^2 = C_1 (\sum_{j=1}^{k} |\rho_n^j (V_j)|^{2^* (s) -2}) |U_n|,
\\
&\partial_{\nu}^{s} D_n^3 = C_1 (|R_n|^{2^* (s) -2})|U_n|.
\end{split}
\right.
\end{equation*}
Then, from \eqref{eq-d-decom} we have $|D_n| \leq D_n^1 + D_n^2 + D_n^3$ by the maximum principle.
Because $\|U_n\|$ is uniformly bounded for $n \in \mathbb{N}$, the Sobolev-trace inequality gives
\begin{equation*}
\sup_n \|U_n (\cdot,0)\|_{L^{2^* (s)}(\Omega)} \leq C \sup_n \| U_n \| \leq C.
\end{equation*}
Since $V^0$ is a bounded, applying Lemma \ref{lem-trace} we have
\begin{equation}\label{eq-d-1}
\| D_n^1 (\cdot,0)\|_{L^{q_1}} \leq C \| U_n (\cdot,0) \|_{L^{2^* (s)}(\Omega)},
\end{equation}
where $q_1$ satisfies $\frac{1}{2^* (s)} -\frac{1}{q_1} = \frac{2s}{N}$. For $1 \leq j \leq k$ we see from Lemma \ref{lem-entire-decay} that $|V_j (\cdot,0)|^{p_n -1} \in L^r$ for any number $r>\frac{N}{4s}$. Hence, we may calculate to see that
\begin{equation*}
\bigl\| \rho_{n}^{j}  (V_j )^{p_n -1}(\cdot,0) \bigr\|_{L^{r}} \leq \lambda_n^{2s -\frac{N}{r}}.
\end{equation*}
Using this we get
\begin{equation}\label{eq-d-2}
\begin{split}
\| D_n^2  (\cdot,0) \|_{L^{q_2}} &\leq C \biggl\| \sum_{j=1}^{k} |\rho_n^j (V^j)^{2^* (s) -2}(\cdot,0)|\biggr\|_{L^{r}} \| U_n (\cdot,0)\|_{L^{2^* (s)}(\Omega)}
\\
& \leq C \lambda_n^{2s- \frac{N}{r}},
\end{split}
\end{equation}
where $q_2$ is such that $N \left( \frac{1}{r} + \frac{N-2s}{2N} - \frac{1}{q_2}\right) = 2s.$ We note that $2s - \frac{N}{r} =\frac{N-2s}{2} - \frac{N}{q_2}$, and it is easy to check that $\frac{N}{N-2s} < q_2 < \frac{2N}{N-2s}$ for $r$ sufficiently close to $\frac{N}{4s}$.
In view of the definition \eqref{eq-def}, the estimates \eqref{eq-d-2} and \eqref{eq-d-1} imply
\begin{equation}\label{eq-d1d2}
\| D_n^1(\cdot,0) \|_{\lambda_n, q_1, q_2} + \| D_n^2 (\cdot,0) \|_{\lambda_n, q_1, q_2} \leq C.
\end{equation}
On the other hand, since $\| R_n \| = o(1)$ we have $\| R_n^{2^*(s) -2} (\cdot,0) \|_{L^{\frac{N}{2s}}(\Omega)} = \|R_n (\cdot,0)\|_{L^{\frac{2N}{N-2s}}(\Omega)}^{\frac{4s}{N-2s}} = o (1).$ Thus, applying  Lemma \ref{lem-harnack-basic} we get
\begin{equation}\label{eq-d3}
\| D_n^3 (\cdot,0) \|_{\lambda_n, q_1, q_2} \leq o (1) \| D_n (\cdot,0) \|_{\lambda_n, q_1, q_2}.
\end{equation}
Combining \eqref{eq-d1d2} and \eqref{eq-d3} we have
\begin{equation*}
\begin{split}
\| D_n (\cdot,0)\|_{\lambda_n, q_1, q_2} & \leq \| D_n^1 (\cdot,0) \|_{\lambda_n, q_1, q_2}+\| D_n^2 (\cdot,0) \|_{\lambda_n, q_1, q_2}+\| D_n^3 (\cdot,0) \|_{\lambda_n, q_1, q_2}
\\
&\leq C + o(1) \| D_n (\cdot,0)\|_{\lambda_n, q_1, q_2},
\end{split}
\end{equation*}
which gives $\| D_n (\cdot,0) \|_{\lambda_n, q_1, q_2} \leq C$ for a constant $C>0$ independent of $n \in \mathbb{N}$. This completes the proof.
\end{proof}
The next lemma is for a bootstrap argument.
\begin{lem}\label{lem-boot}
Consider two numbers $q_1$ and $q_2$ such that $\frac{N+2s}{N-2s} < q_2 < \frac{2N}{N-2s} <q_1 < \frac{N}{2s}\frac{N+2s}{N-2s}$. Let $\gamma_1$ and $\gamma_2$ satisfy
\begin{equation*}
\frac{1}{\gamma_i} = \frac{N+2s}{N-2s} \frac{1}{q_i} -\frac{2s}{N}, ~ i=1,2.
\end{equation*}
Assume that for some $v \in L^{q_2}(\Omega)$, $U \in H^1_0(t^{1-2s},\mathcal{C})$ solves
\begin{equation*}
\left\{\begin{split}
\textrm{\em div}(t^{1-2s}U) = 0 &\quad \textrm{in}~\mathcal{C},
\\
U=0 &\quad \textrm{on}~\partial_{L}\mathcal{C},
\\
\partial_{\nu}^{s} U \leq |v|^{2^{*}(s) -1} + A& \quad \textrm{on}~\Omega \times \{0\}.
\end{split}\right.
\end{equation*}
Then there is a constant $C=C(q_1, q_2, \Omega)$ such that
\begin{equation*}
\| U ( \cdot, 0) \|_{\lambda, \gamma_1, \gamma_2} \leq C \left( \|v\|_{\lambda, q_1, q_2}^{2^{*}(s)-1} +1 \right).
\end{equation*}
\end{lem}
\begin{proof}
Consider two positive functions $v_1 \in L^{q_1}(\Omega)$ and $v_2 \in L^{q_2} (\Omega)$ such that $|v| \leq v_1 + v_2$. Then,
\begin{equation*}
\partial_{\nu}^s U \leq C \left( v_1^{2^* (s)-1} + v_2^{2^* (s)-1} +1\right).
\end{equation*}
Let $U_1 \in H^1_0(t^{1-2s},\mathcal{C})$ and $U_2 \in H^1_0(t^{1-2s},\mathcal{C})$ be solutions to
\begin{equation}\label{eq-u-v}
\left\{ \begin{array}{ll}
\textrm{div}(t^{1-2s} \nabla U_i ) = 0 &\quad \textrm{in}~\mathcal{C},
\\
\partial_{\nu}^s U_i = v_i^{2^*(s)-1}&\quad \textrm{on}~\Omega \times \{0\},
\end{array}
\right. \quad \textrm{for}~ i =1,2.
\end{equation}
We note that $U_i$ is nonnegative.
Multiplying \eqref{eq-u-v} by $U_i^{\beta}$ for some $\beta > 1$, we have
\begin{equation*}
\frac{4\beta}{(\beta+1)^2} \int_{\mathcal{C}} t^{1-2s} |\nabla (U_i^{(\beta+1)/2})|^2\, dx dt = \int_{\Omega \times \{0\}} v_i^{2^*(s)-1}(x)U_i^{\beta}(x,0)\, dx.
\end{equation*}
Now we apply the Sobolev-trace inequality and H\"older's inequality to get
\begin{equation*}
\| U_i^{\frac{\beta+1}{2}} (x,0)\|_{L^{\frac{2N}{N-2s}}(\Omega)} \leq C \| v^{2^*(s)-1} \|_{L^r} \| U_i^{\beta} \|_{L^{\frac{\beta+1}{2\beta}\frac{2N}{N-2s}}},
\end{equation*}
where $r$ is chosen to satisfy $\frac{1}{r} + \frac{2\beta}{\beta+1} \frac{N-2s}{2N} = 1$. We take $\beta$ satisfying $\gamma_i =\frac{\beta+1}{2}\frac{2N}{N-2s}$.
Then one has $(2^*(s)-1)r = q_i$, and so the above inequality gives
\begin{equation*}
\| U_i (\cdot,0) \|_{L^{\gamma_i}} \leq C \| v_i \|_{L^{q_i}}^p\quad \textrm{for}~ i =1,2.
\end{equation*}
Thus we get
\begin{equation}\label{eq-boot}
\begin{split}
\| U(\cdot,0)\|_{\lambda, \gamma_1, \gamma_2} & \leq \| U_1 (\cdot, 0)\|_{L^{\gamma_1}} + \lambda^{\frac{N}{\gamma_2} -\frac{N}{2^* (s)}} \| U_i (\cdot, 0)\|_{L^{\gamma_2}} + C
\\
&\leq \| v_1 \|_{L^{q_1}}^{2^*(s)-1} + \lambda^{\frac{N}{\gamma_2} -\frac{N}{2^*(s)}} \|v_2\|_{L^{q_2}}^{2^* (s)-1} + C.
\end{split}
\end{equation}
From the fact that
$\frac{1}{2^* (s) -1 } \left(\frac{N}{\gamma_2}-\frac{N}{2^* (s)}\right) = \frac{N}{q_2} -\frac{N}{2^{*}(s)},$
the estimate \eqref{eq-boot} implies
\begin{equation*}
\left\| U(\cdot,0)\right\|_{\lambda,\gamma_1,\gamma_2} \leq C \left( \left\| v\right\|_{\lambda,q_1,q_2}^{2^* (s) -1} +1\right),
\end{equation*}
which shows the lemma.
\end{proof}
\begin{proof}[Proof of Proposition \ref{prop-q1-q2}]
By the result of Lemma \ref{lem-start}, there exists two numbers $q_{1} \in \left(\frac{2N}{N-2s},\infty\right)$ and $q_{2} \in \left( \frac{N}{N-2s},\frac{2N}{N-2s}\right)$ satisfying
\begin{equation*}
\sup_{n \in \mathbb{N}} \| D_n (\cdot,0) \|_{\rho_n, q_{1}, q_{2}} \leq C.
\end{equation*}
Then, by Lemma \ref{lem-boot} we have
\begin{equation*}
\sup_{n \in \mathbb{N}} \|D_n (\cdot,0)\|_{\rho_n, \gamma_1, \gamma_2} \leq C,
\end{equation*}
where $\gamma_1$ and $\gamma_2$ satisfy $\frac{1}{\gamma_i} = \frac{N+2s}{N-2s}\frac{1}{q_i} - \frac{2s}{N}$ for $i=1,2$. Iteratively applying this process with H\"older's inequality,
one can conclude the desired result.
\end{proof}

\section{Integral estimates}
In this section we establish some sharp $L^q$ estimates for solution sequence $\{U_n\}$ on some suitable annuli around the slowest bubbling point $x_n$, which play a fundamental role to prove our main theorems.
Let us define several domains:
\begin{itemize}
\item $B^{N}(x,r) = \{ y \in \mathbb{R}^N: |x-y |\leq r \}$ for $ x\in \mathbb{R}^N$ and $r>0$.
\\
\item $B^{N+1}(x,r) = \{ z \in \mathbb{R}^{N+1}_{+}: |z-(x,0)|\leq r\}$ for $x \in \mathbb{R}^{N}$ and $r>0$.
\\
\item For $d = N, N+1$, $A^{d}(x,[r_1,r_2]) = B^{d}(x,r_2) \setminus B^{d}(x,r_1) $ for $x \in \mathbb{R}^d$ and $r_2 >r_1 >0$.
\item For a domain $D \in \mathbb{R}^{N+1}_{+}$
\\
$\partial_{+} D = \{ (x,t) \in \partial D : t>0\}$,
\\
$\partial_{b}D = \{ x \in \mathbb{R}^N : (x,0) \in \partial D \cap \mathbb{R}^N \times \{0\}\}.$
\end{itemize}

Consider the annuli $A^N ({x_n}, [5m \lambda_n^{-1/2}, (5m + 5)\lambda_n^{-1/2}])$, $1 \leq m \leq k+1$. \
By choosing a subsequence, we may assume that for some $m \in \{1,\cdots, k+1\}$, the annuli $A^N({x_n}, [5m \lambda_n^{-1/2}, 5(m+1)\lambda_n^{-1/2}]) $ does not contain any other bubbling points. Let
\begin{equation*}
\left\{\begin{array}{l}
\mathcal{A}_n^1 (d) = A^d ({x_n}, [(5m+1)\lambda_n^{-1/2}, (5m+4)\lambda_n^{-1/2}]) \cap \mathcal{C} ~\textrm{or}~ \Omega,
\\
\mathcal{A}_n^2 (d) = A^d ({x_n}, [(5m+2)\lambda_n^{-1/2}, (5m+3)\lambda_n^{-1/2}])  \cap \mathcal{C}~\textrm{or}~ \Omega,
\end{array}
\right.\quad \textrm{for}\quad n \in \mathbb{N}, \quad d = N, N+1.
\end{equation*}
For a measurable set $A \subset \mathbb{R}^{n+1}_{+}$ we define a weighted measure
\begin{equation}\label{eq-wm}
m_s (A) = \int_A t^{1-2s} dx dt,
\end{equation}
and a weighted average
\begin{equation}
\fint_A f (x,t) t^{1-2s} dx dt = \frac{\int_A f(x,t) t^{1-2s} dxdt}{\int_A t^{1-2s} dx dt}.
\end{equation}
Now we state the result on the integral esimates of $U_n$ on the annuli $\mathcal{A}^1_n(N)$ and $\mathcal{A}^1_n(N+1)$.
\begin{prop}\label{int-estimate}
Let $\{ U_n\}_{n=1}^{\infty}$ be a sequence of solutions to \eqref{eq-main-q}
with $p = p_n \to 2^*(s)-2$ such that $\|U_n\| < C$ for some $C>0$ independent of $n \in \mathbb{N}$.
Then, for any $q >1$, there exists a constant $C_q >0$ such that
\begin{equation}\label{eq-average-q}
\sup_{n \in \mathbb{N}} \left\{ \fint_{\mathcal{A}_n^1 (N+1)} |U_n (x,t)|^q t^{1-2s} dxdt + \fint_{\mathcal{A}_n^1 (N)} |U_n (x,0)|^q dx \right\} \leq C_q.
\end{equation}
\end{prop}
To prove this proposition, we need the following lemma.
\begin{lem}\label{lem-average}
For $f \geq 0$, assume that $U \in H_{0}^{1}(t^{1-2s},\mathcal{C})$ satisfies
\begin{equation*}
\left\{\begin{array}{ll}
\textrm{\em div}(t^{1-2s} \nabla U) = 0&\quad \textrm{in}~\mathcal{C},
\\
\partial_{\nu}^s U = f &\quad \textrm{on}~\Omega \times \{0\},
\\
U=0 &\quad \textrm{on}~\partial_{L}\mathcal{C}.
\end{array}
\right.
\end{equation*}
For $\gamma \in \left(1, \frac{2N+2}{2N+1}\right)$, there exists a constant $C_q >0$ such that
\begin{equation*}
\left( \fint_{B^{N+1}(x,r)} t^{1-2s} U^{\gamma} dx dt \right)^{1/\gamma} \leq \fint_{B^{N+1}(x,1)} t^{1-2s} U^{\gamma} dx dt + C_q \int^{1}_{r} \left( \frac{1}{\rho^{N-2s}} \int_{B^N (x,\rho)} f(y) dy \right) \frac{d\rho}{\rho}
\end{equation*}
holds for any $x \in \Omega$ and $r \in (0,r_0)$ where $r_0= \textrm{dist}(x,\partial \Omega)$.
\end{lem}
\noindent This lemma is analogous to Proposition C.1 in \cite{CPY}. We refer to \mbox{Appendix A} for the proof of this result.

\begin{proof}[Proof of Proposition \ref{int-estimate}] We consider the function $D_n$ such that
\begin{equation}
\left\{\begin{array}{ll}
\textrm{div}(t^{1-2s} \nabla D_n) = 0&\quad \textrm{in}~ \mathcal{C},
\\
D_n =0&\quad \textrm{on}~\partial_L \mathcal{C},
\\
\partial_{\nu}^{s} D_n = |U_n|^{2^{*}(s)-1} + C&\quad \textrm{on}~\Omega \times\{0\}.
\end{array}
\right.
\end{equation}
Then we have $\|D_n \| \leq C \|U_n \| +C$, and also $|U_n | \leq D_n$ by the maximum principle.
Choose a point $y \in \Omega$. For $\gamma \in \left( 1, \frac{2N+2}{2N+1}\right)$ we claim that
\begin{equation}\label{eq-claim}
\sup_{ r \in ( \lambda_n^{-1/2}, 1)} \fint_{B^{N+1}(y,r)} t^{1-2s} |D_n|^{\gamma} (x,t) dx dt \leq C,
\end{equation}
with $C>0$ independent of $y \in \Omega$ and $n \in \mathbb{N}$. We first note that
\[\sup_{n \in \mathbb{N}}\|D_n \| \leq  C \sup_{n \in \mathbb{N}}\| U_n\| + C \leq C.\]
Thus, using the Sobolev embedding \eqref{eq-sobolev-weight} and H\"older's inequality we deduce
\begin{equation*}
\fint_{B^{N+1} (y,1)} t^{1-2s} |D_n|^{\gamma}(x,t) dxdt \leq C.
\end{equation*}
Combining this with Lemma \ref{lem-average}, for  each $0<r<\textrm{dist}(y,\partial \Omega)$ we get
\begin{equation}\label{eq-claim-proof}
\left( \fint_{B^{N+1} (y,r)} t^{1-2s} D_n^{\gamma} dx dt\right)^{1/\gamma} \leq %\fint_{B^{N+1} (y,1)} t^{1-2s} U_n^{\gamma} dx dt
C + C \int_{r}^{1} \left[ \frac{1}{\rho^{N-2s}} \int_{B^N (y,\rho)} \left(|U_n|^{2^* (s)-1} (x,0) + C \right) dx \right] \frac{d\rho}{\rho}.
\end{equation}
In order to bound the last term on the right, we set $q_1 = \frac{N(N+2s)}{s(N-2s)} $ and $q_2 = \frac{N+2s}{N-2s}$, and apply  Proposition \ref{prop-q1-q2} to find functions $w_n^1 \in L^{q_1}(\Omega)$ and $w_n^2 \in L^{q_2}(\Omega)$ such that $|U_n | \leq w_n^1 + w_n^2$ and
\begin{equation}\label{eq-w1w2}
\| w_n^1 \|_{L^{q_1}} \leq C \quad\textrm{and}\quad \| w_n^2 \|_{L^{q_2}} \leq C \lambda_n^{N/{2^* (s)} - N/q_2}.
\end{equation}
Then, %by the triangle inequality
\begin{equation}\label{eq-tri}
\begin{split}
&\int_{\sigma_n^{-1/2}}^{1} \frac{1}{t^{N-2s+1}}\left[\int_{B_t (x_n)} U_n^{2^* (s) -1} (y,0) dy \right]dt
\\
&\quad\quad \leq C \int_{r}^1 \frac{1}{t^{N-2s+1}} \left[ \int_{B^N (y,t)} (w_n^1)^{2^* (s)-1} (x) dx \right] dt + C \int_{r}^{1} \frac{1}{t^{N-2s+1}} \left[ \int_{B^N (y,t)} (w_n^2)^{2^* (s) -1} (x) dx \right] dt.
\end{split}
\end{equation}
We use \eqref{eq-w1w2} to deduce
\begin{equation*}
 \int_{r}^1 \frac{1}{t^{N-2s+1}} \left[ \int_{B^N (y,t)} (w_n^1)^{2^* (s) -1} (x) dx \right] dt  \leq C \int_{\sigma_n^{-1/2}}^{1} \frac{1}{t^{N-s}} (t^{N (N-2s+1)/N})\| (w_n^1)^{2^* (s) -1} \|_{L^{\frac{N}{s}}(\Omega) } \leq C,
\end{equation*}
and
\begin{equation*}
\begin{split}
&\int_{r}^{1} \frac{1}{t^{N-2s+1}} \left[ \int_{B^N (y,t)} (w_n^2)^{2^* (s) -1} (x,0) dx \right] dt
\\
&\quad\quad\quad \leq \int_{\sigma_n^{-1/2}}^{1} \frac{1}{t^{N-2s+1}} \left[ C \sigma_n^{\frac{N-2s}{2}- \frac{N (N-2s)}{N+2s}} \right]^{\frac{N+2s}{N-2s}} dt \leq C \sigma_n^{(N-2s)/2} \sigma_n^{-(N-2s)/2} = C.
\end{split}
\end{equation*}
These two estimates with \eqref{eq-tri} and \eqref{eq-claim-proof} prove the claim \eqref{eq-claim}. As a result we have
\begin{equation}\label{eq-gamma}
\sup_{n \in \mathbb{N}}  \fint_{A_n^{N+1}} |U_n (x,t)|^{\gamma} t^{1-2s} dxdt  \leq C.
\end{equation}
To complete the proof, we only need to raise $\gamma$ to higher orders in the above average estimate. In this regard, we set
\begin{equation*}
\widetilde{U}_n (z) = U_n (\lambda_n^{-\frac{1}{2}} z + (x_n,0)).
\end{equation*}
Then it satisfies
\begin{equation*}
\left\{ \begin{array}{ll} \textrm{div}(t^{1-2s} \nabla \widetilde{U}_n) = 0, &\quad \textrm{in} ~ B^{N+1}(0,5m+5) %  \mathcal{C}_n = \{ z \in \mathbb{R}^{n+1}_{+}: \lambda_n^{-1/2} z \in \mathcal{C}\},
\\
\partial_{\nu}^s \widetilde{U}_n = \lambda ^{-s} (\widetilde{U}_n^{p-1} + C) \widetilde{U}_n
&\quad \textrm{on}~ B^N (0,5m+5) \times\{0\},%\Omega_n  \times \{0\} = \{ (x,0) \in \mathbb{R}^{n+1}: \lambda_n^{-1/2} x \in \Omega\},
\end{array}
\right.
\end{equation*}
and for $\gamma \in \left(1,\frac{2N+2}{2N+1}\right)$, the estimate \eqref{eq-gamma} gives
\begin{equation}\label{eq-gamma-1}
\int_{A^{N+1}(0,[5m, 5m+5])} t^{1-2s} \widetilde{U}_n^{\gamma} dx dt \leq C.
\end{equation}
Moreover, since $A^N (x_n, [5m \lambda_n^{-1/2}, 5(m+1) \lambda_n^{-1/2}])$ does not any bubbling point of $U_n$, we easily get
\begin{equation*}
\lim_{n \rightarrow \infty} \int_{A^{N+1}(0, [5m+1, 5m+4])} \widetilde{U}_n (x,0)^{2^* (s)} dx = 0.
\end{equation*}
Given this and \eqref{eq-gamma-1}, we may apply Lemma \ref{lem-harnack} to deduce that for any $q>1$,
\begin{equation*}
\int_{A^{N+1}(0,[5m+2,5m+3])} t^{1-2s}\widetilde{U}_n^q dx dt + \int_{A^{N}(0,[5m+2,5m+3])} \widetilde{U}_n^q dx \leq C_q.
\end{equation*}
By writing down this inequality in terms of $U_n$ on $\mathcal{A}_n^{N+1}$ and $\mathcal{A}_n^N$, we get the desired inequality \eqref{eq-average-q}. The proof is completed.
\end{proof}

\begin{prop}\label{int-grad-estimate}
Let $\{U_n\}_{n\in \mathbb{N}}$ be a sequence of solutions to \eqref{eq-main-q} with $p = p_n \to 2^*(s)-2$ such that
$\|U_n\|$ is bounded uniformly for $n \in \mathbb{N}$. Then there exists $C > 0$ independent of $n$ such that
\[
\int_{\mathcal{A}^2_n(N+1)} t^{1-2s}|\nabla U_n(x,t)|^2\,dxdt \leq C\lambda_n^{\frac{2s-N}{2}}
\]
\end{prop}
\begin{proof}
Let $\phi_n \in C^\infty_0(A^{N+1} ({x_n}, [(5m+1)\lambda_n^{-1/2}, (5m+4)\lambda_n^{-1/2}]))$ be a sequence of cut-off functions such that
$\phi_n = 1$ on $A^{N+1} ({x_n}, [(5m+2)\lambda_n^{-1/2}, (5m+3)\lambda_n^{-1/2}])$ and
$0 \leq \phi_n \leq 1$, $|\nabla\phi_n| \leq C\lambda_n^{1/2}$ on $A^{N+1} ({x_n}, [(5m+1)\lambda_n^{-1/2}, (5m+4)\lambda_n^{-1/2}])$.
Then we see from \eqref{eq-main-q} that
\begin{multline}
\int_{\mathcal{A}^1_n(N+1)}t^{1-2s}\nabla U_n(x,t)\cdot\nabla\left(\phi_n^2(x,t)U_n(x,t)\right)\,dx dt \\
\leq C_s\int_{\mathcal{A}^1_n(N)}\left(|U_n(x,0)|^{p_n+1}+\mu|U_n (x,0)|\right)|\phi_n^2(x,0)U_n(x,0)|\,dx,
\end{multline}
which yields
\[
\begin{aligned}
&\int_{\mathcal{A}^1_n(N+1)}t^{1-2s}\phi^2_n(x,t)|\nabla U_n(x,t)|^2\,dxdt \\
&\leq C\int_{\mathcal{A}^1_n(N)}|U_n(x,0)|^{p_n+2} +|U_n(x,0)|^2\,dx
+C\int_{\mathcal{A}^1_n(N+1)}t^{1-2s}|U_n(x,t)\nabla\phi(x,t)|^2\,dxdt \\
&\leq C\int_{\mathcal{A}^1_n(N)}\left(|U_n(x,0)|^{2^*(s)} +|U_n(x,0)|^2 + 1\right)\,dx
+C\lambda_n^{1}\int_{\mathcal{A}^1_n(N+1)}t^{1-2s}|U_n(x,t)|^2\,dxdt.
\end{aligned}
\]
Then, this and Proposition \ref{int-estimate} show that
\[
\int_{\mathcal{A}^2_n(N+1)}t^{1-2s}|\nabla U_n (x,t)|^2\,dxdt
\leq C\lambda_n^{-\frac{N}{2}} +C\lambda_n^{-\frac{N+2-2s}{2}+1} \leq C\lambda_n^{\frac{2s-N}{2}}.
\]
The proof is completed.
\end{proof}

\section{End of the proofs of main theorems}
We shall complete in this section the proof of Theorems \ref{thm-main-1} and \ref{thm-uniform-bound}.
As we explained before, the strategy for the proof of Theorem \ref{thm-uniform-bound} is to show there could be no bubbles in the decomposition \eqref{eq-decom} for any uniformly norm bounded sequence of solutions to \eqref{eq-main-q} with $p = p_n \to 2^*(s)-2$.
Indeed, we will show a contradiction takes place if we assume that there are bubbles.
This will be accomplished by using a local Pohozaev identity on concentric balls centered the bubbling point $x_n$, the blow up rate of which is minimal among all bubbling points.

\begin{proof}[Proof of Theorem \ref{thm-uniform-bound}]
Assume that $\{U_n\}_{n \in \mathbb{N}}$ is noncompact. Then we recall that the solutions follow the representation
\[
U_n = V^0 + \sum_{j=1}^{k}\rho_{n}^{j}( V^{j}) + R_n,
\]
described in Lemma \ref{lem-cc-bounded} with some $R_n \to 0$ in $H_{0}^{1}(t^{1-2s},\mathcal{C})$. We also may assume that our slowest bubbling point
$x_n$ is $x_n^1$.
We denote
\[
\mathcal{E}_n(N, l) = B^{N}(x_n, l\lambda_n^{-1/2}) \cap \Omega,\quad
\mathcal{E}_n(N+1, l) = B^{N+1}((x_n,0), l\lambda_n^{-1/2}) \cap \mathcal{C}
\]
where $l \in (5m+2, 5m+3)$.
By the local Pohozaev identity \eqref{eq-local}, we have
\begin{equation}\label{eq-poho-main}
\begin{aligned}
&C_s \left\{ \left(\frac{N}{p_n +2} - \frac{N-2s}{2}\right) \int_{\mathcal{E}_n(N, l)} |U_n(x,0)|^{p_n +2} dx + \mu s\int_{\mathcal{E}_n(N,l)} |U_n(x,0)|^2 dx\right\} \\
&\quad = \int_{\partial\mathcal{E}_n(N, l)} \left(\frac{\mu}{2}|U_n(x,0)|^2 +\frac{1}{p_n+2}|U_n(x,0)|^{p_n+2}\right)(x-x_0, \nu_x)\,dS_x \\
&\quad\quad+\int_{\partial_+\mathcal{E}_n(N+1,l)} t^{1-2s} \left((z-z_0, \nabla U_n(z)) \nabla U_n(z)
-(z-z_0) \frac{|\nabla U_n(z)|^2}{2}, \nu_z \right)\,dS_z \\
&\quad\quad+\left(\frac{N-2s}{2}\right)\int_{\partial_+\mathcal{E}_n(N+1,l)} t^{1-2s} U_n(z)\frac{\partial U_n(z)}{\partial \nu_z}\,dS_z,
\end{aligned}
\end{equation}
where $x_0 \in \R^N$ is arbitrary, $z_0 = (x_0, 0)$ and $z = (x,t)$.
We decompose $\partial\mathcal{E}_n(N,l)$ as
\[
\partial\mathcal{E}_n(N,l) =
\partial_{\text{int}}\mathcal{E}_n(N,l) \cup \partial_{\text{ext}}\mathcal{E}_n(N,l)
\]
where
$\partial_{\text{int}}\mathcal{E}_n(N,l) := \partial\mathcal{E}_n(N,l) \cap \Omega$ and
$\partial_{\text{ext}}\mathcal{E}_n(N,l) := \partial\mathcal{E}_n(N,l) \cap \partial\Omega$.
Similarly,
\[
\partial_+\mathcal{E}_n(N+1,l) =
\partial_{\text{int}}\mathcal{E}_n(N+1,l) \cup \partial_{\text{ext}}\mathcal{E}_n(N+1,l)
\]
where
$\partial_{\text{int}}\mathcal{E}_n(N+1,l) := \partial_+\mathcal{E}_n(N+1,l) \cap \mathcal{C}$ and
$\partial_{\text{ext}}\mathcal{E}_n(N+1,l) := \partial_+\mathcal{E}_n(N+1,l) \cap \partial\mathcal{C}$.
For each $x_n$ and $l$, we have two cases:
\[
\text{(i) } B^N(x_n, l) \subset \Omega \quad \text{or} \quad \text{(ii) } B^N(x_n, l) \not\subset \Omega.
\]
For the case (i), we take $x_0 = x_n$. For the case (ii), we take $x_0 \in \R^N \setminus \Omega$
such that $|x_0 - x_n| \leq C\lambda_n^{-1/2}$ and $\nu_x \cdot (x-x_0) \leq 0$
at all $x \in \partial_{\text{ext}}\mathcal{E}_n(N,l)$.
Then, we see from the fact $\nu_z = (\nu_x, 0)$ that
\[
\nu_z\cdot(z-z_0) = (\nu_x, 0) \cdot (x-x_0, t-0) = \nu_x\cdot(x-x_0) \leq 0
\]
for any $z = (x, t) \in \partial_{\text{ext}}\mathcal{E}_n(N+1, l)$.
Then, the fact $u_n = 0$ on $\partial_{\text{ext}}\mathcal{E}_n(N, l) \cup \partial_{\text{ext}}\mathcal{E}_n(N+1, l)$ yields
\begin{align*}
&\int_{\partial_\text{ext}\mathcal{E}_n(N,l)} \left(\frac{\mu}{2}|U_n(x,0)|^2 +\frac{1}{p_n+2}|U_n(x,0)|^{p_n+2}\right)(x-x_0, \nu_x)\,dS_x = 0, \\
&\int_{\partial_\text{ext}\mathcal{E}_n(N+1,l)} t^{1-2s} U_n(z)\frac{\partial U_n(z)}{\partial \nu_z}\,dS_z = 0.
\end{align*}
Also, since $\nabla U_n = \pm|\nabla U_n|\nu_z$ on $\partial_\text{ext}\mathcal{E}_n(N+1,l)$, we see
\begin{align*}
&\int_{\partial_\text{ext}\mathcal{E}_n(N+1,l)} t^{1-2s} \left((z-z_0, \nabla U_n(z)) \nabla U_n(z)
-(z-z_0) \frac{|\nabla U_n(z)|^2}{2}, \nu_z \right)\,dS_z, \\
& = \int_{\partial_\text{ext}\mathcal{E}_n(N+1,l)} t^{1-2s}\frac{|\nabla U_n(z)|^2}{2}\left(z-z_0, \nu_z\right) \,dS_z \leq 0.
\end{align*}
Combining this with \eqref{eq-poho-main}, we obtain
\begin{equation}\label{eq-poho-est}
\begin{aligned}
\int_{\mathcal{E}_n(N,l)} |U_n(x,0)|^2\,dx &\leq
C\lambda_n^{-1/2}\int_{\partial_\text{int}\mathcal{E}_n(N,l)} \left(|U_n(x,0)|^2 +|U_n(x,0)|^{p_n+2}\right)\,dS_x \\
&+C\int_{\partial_\text{int}\mathcal{E}_n(N+1,l)} t^{1-2s} |U_n(z)||\nabla U_n(z)|\,dS_z \\
&+C\lambda_n^{-1/2}\int_{\partial_\text{int}\mathcal{E}_n(N+1,l)} t^{1-2s}|\nabla U_n(z)|^2\,dS_z.
\end{aligned}
\end{equation}
Extending $U_n$ to $0$ on $\R^{N+1}\setminus \mathcal{C}$ and integrating \eqref{eq-poho-est} with respect to $l$, we get
\begin{equation*}
\begin{aligned} \int_{5m+2}^{5m+3}\int_{\mathcal{E}_n(N,l)} |U_n(x,0)|^2\,dx\,dl &\leq
C\int_{\mathcal{A}_n^2(N)} \left(|U_n(x,0)|^2 +|U_n(x,0)|^{p_n+2}\right)\,dx \\
&+C\lambda_n^{1/2}\int_{\mathcal{A}_n^2(N+1)} t^{1-2s} |U_n(z)||\nabla U_n(z)|\,dz \\
&+C\int_{\mathcal{A}_n^2(N+1)} t^{1-2s}|\nabla U_n(z)|^2\,dz,
\end{aligned}
\end{equation*}
from which we deduce that
\begin{equation}\label{eq-poho-upper}
\int_{\mathcal{E}_n(N,(5m+2)\lambda_n^{-1/2})}|U_n (x,0)|^2\,dx \leq \int_{5m+2}^{5m+3}\int_{\mathcal{E}_n(N,l)} |U_n (x,0)|^2\,dx\,dl \leq C\lambda_n^{\frac{2s-N}{2}},
\end{equation}
by applying Proposition \ref{int-estimate}, Proposition \ref{int-grad-estimate} and H\"older inequality.

On the other hand, one can observe by extending $U_n = 0$ on $\R^{N+1}_+ \setminus \Omega$ that for large $n$
\begin{equation*}
\begin{split}
&\int_{\mathcal{E}_n(N, (5m+2)\lambda_n^{-1/2})} |U_n (x,0)|^2\,dx \\
&=\int_{B^N(x_n,(5m+2)\lambda_n^{-1/2})} |U_n (x,0)|^2\, dx
\geq \int_{B^N(x_n,\lambda_n^{-1})} |U_n (x,0)|^2\, dx \\
&\geq C\int_{B^N (x_n, \lambda_n^{-1})} |\rho_n^1(V^1)(x,0)|^2\,dx \\
&\qquad\qquad-C\int_{B^N(x_n,\lambda_n^{-1})}\sum_{j=2}^k|\rho_n^j(V^j)(x,0)|^2+|V^0(x,0)|^2 +|R_n(x,0)|^2\,dx.
\end{split}
\end{equation*}
One can compute
\[
\int_{B^N (x_n, \lambda_n^{-1})} |\rho_n^1(V^1)(x,0)|^2\,dx =
\left(\int_{B^N (0, 1)} |V^1(x,0)|^2\,dx\right)\lambda_n^{-2s}
\]
and
\[
\begin{aligned}
\int_{B^N(x_n,\lambda_n^{-1})}|\rho_n^j(V^j)(x,0)|^2\,dx &=
\left(\int_{S_n^j} |V^j(x,0)|^2\,dx\right)(\lambda_n^j)^{-2s} \\
&=\left(\int_{S_n^j} |V^j(x,0)|^2\,dx\right)\left(\frac{\lambda_n^j}{\lambda_n}\right)^{-2s}\lambda_n^{-2s},
\end{aligned}
\]
where
\[
S_n^j := \lambda_n^j(B^N(x_n,\lambda_n^{-1}) - x_n^j).
\]
Then, the fact
\[
\frac{\lambda_n^j}{\lambda_n} +\lambda_n\lambda_n^j |x_n-x_n^j|^2 \rightarrow \infty~\textrm{as}~ n \rightarrow \infty~\textrm{for all}~ j \neq 1,
\]
implies that
\[
\left(\int_{S_n^j} |V^j(x,0)|^2\,dx\right)\left(\frac{\lambda_n^j}{\lambda_n}\right)^{-2s} = o(1).
\]
Also, since $V^0 \in L^\infty(\mathcal{C})$ and $R_n = o(1)$ in $H_{0}^{1}(t^{1-2s},\mathcal{C})$ as $n \to \infty$,
we see
\[
\int_{B^N (x_n, \lambda_n^{-1})} |V^0(x,0)|^2\,dx \leq C\lambda_n^{-N} \leq o(1)\lambda_n^{-2s}
\]
and
\[
\int_{B^N (x_n, \lambda_n^{-1})}|R_n(x,0)|^2\,dx \
\leq C\left(\int_{\Omega}|R_n(x,0)|^{2^*(s)}\,dx\right)^{\frac{2}{2^*(s)}} {\lambda_n}^{-2s}
=o(1){\lambda_n}^{-2s}
\]
from the Sobolev-trace inequality \eqref{eq-sharp-trace}.
Thus we deduce
\begin{equation}\label{eq-poho-lower}
\int_{\mathcal{E}_n(N, (5m+2)\lambda_n^{-1/2})} |U_n (x,0)|^2\,dx
\geq c\lambda_n^{-2s}.
\end{equation}
Now, combining \eqref{eq-poho-upper} with
\eqref{eq-poho-lower} we finally obtain
\begin{equation*}
\lambda_n^{-2s} \leq C\lambda_n^{\frac{2s-N}{2}}.
\end{equation*}
Since $\lim_{n \rightarrow \infty} \lambda_n = \infty$, this inequality implies that $-2s \leq \frac{2s-N}{2}$, which is equivalent to $N \leq 6s$. However this contradicts with our assumption $N>6s$.
Thus, one can conclude that there are no bubbles in $U_n$
so that $U_n \rightarrow V^0$ in $H_{0}^{1}(t^{1-2s},\mathcal{C})$, and the set of solutions $\{U_n\}_{n \in \mathbb{N}}$ is compact. This completes the whole proof of Theorem \ref{thm-uniform-bound}.

%Then, by a Kato-Ponce type estimate (see e.g. \cite{T2}), we may deduce that $\{u_n\}_{n \in \mathbb{N}}$ are uniformly bounded in $L^{\infty}(\Omega)$. Now

%Near $x_0$ we have
%\begin{equation*}
%U(x) \sim \frac{\lambda^{\frac{n-2s}{2}}}{(1+ \lambda |x-x_0|)^{n-2s}}.
%\end{equation*}
%Therefore, for $|x-x_0| \sim \lambda^{-1/2}$, we have
%\begin{equation*}
%U(x) \sim 1 \quad \textrm{and}\quad \nabla U(x) \sim \lambda^{1/2}.
%\end{equation*}
%Note that
%\begin{equation*}
%\int_{|x| \leq \lambda^{-1/2}} \frac{\lambda^{(n-2s)}}{(1+\lambda |%x|)^{2(n-2s)}} dx = \int_{|y| \leq 1} \frac{\lambda^{n-2s} \lambda^{-n/2}}{(1+ \lambda^{1/2} |y|)^{2(n-2s)}} dy \sim \lambda^{-2s}.
%\end{equation*}
%On the other hand, the right hand side of the local Pohozaev identity is bounded by
%\begin{equation*}
%\lambda^{-n/2} \lambda^{-1/2 (1-2s)} \lambda = \lambda^{-n/2}\lambda^{1/2 +s}.
%\end{equation*}
%Therefore it should hold that
%\begin{equation*}
%\lambda^{-2s} \leq C \lambda^{-n/2} \lambda^{1/2 +s}.
%\end{equation*}
%Thus it should hold that
%\begin{equation*}
%n< 1+6s.
%\end{equation*}
\end{proof}

\begin{proof}[Proof of Theorem \ref{thm-main-1}]
We use the variational methods and a topological index theory to construct infinitely many solutions to \eqref{eq-main}.
We have already seen that \eqref{eq-main} is equivalent to \eqref{u0inc}.
So let us define
\begin{equation}\label{functional}
I_{\ep}(u) := \frac{1}{2} \int_{\mathcal{C}} t^{1-2s} |\nabla U|^2 dx dt - \frac{\mu}{2} \int_{\Omega} |U(x,0)|^2 dx - \frac{1}{2^*(s)-\ep} \int_{\Omega} |U(x,0)|^{2^*(s)-\ep} dx,
\end{equation}
which is a variational functional for \eqref{eq-main-q}. Then, a variational functional for \eqref{u0inc} corresponds to \eqref{functional} with $\ep = 0$.

For a closed $\mathbb{Z}_2$ invariant set $X \subset H_{0}^{1}(t^{1-2s},\mathcal{C})$,
we denote by $\gamma (X)$ the topological genus of $X$ which stands for the smallest integer $m$
such that there is an odd map $\phi \in C (X, \mathbb{R}^m \setminus \{0\}).$
For $k \in \mathbb{N}$ we define a family of sets $F_k$ by
\begin{equation}
F_k = \{ X \subset H_{0}^{1}(t^{1-2s},\mathcal{C}): X \textrm{ is compact, $\mathbb{Z}_2$-invariant, and $\gamma (X) \geq k$\}}.
\end{equation}
Consider the minimax value $c_{k,\ep}= \inf_{X \in F_k} \max_{u \in X} I_{\ep} (u)$. Then for any small $\ep > 0$, $c_{k,\ep}$ is a critical value of $I_{\ep} (u)$, i.e.,
there exists a solution $u_{k, \ep}$ to \eqref{eq-main-q} such that $c_{\ep, k} = I_\ep(u_{k,\ep})$ (see e.g. \cite[Corollary 7.12]{G}).
It is also well known that $c_{k,\ep} \to \infty$ as $k \to \infty$.

We first show that for each fixed $k \in \mathbb{N}$, $c_{k,\ep}$ is uniformly bounded for $\ep >0$. For this we set
\begin{equation}
A_k := \inf_{X \in F_k} \max_{u \in X} \left[ \frac{1}{2} \int_{\mathcal{C}} t^{1-2s} |\nabla U|^2 dx dt - \frac{\mu}{2} \int_{\Omega} |U(x,0)|^2 dx - \frac{1}{2^*(s)} \int_{\Omega} |U(x,0)|^{\sigma} dx\right],
\end{equation}
where $\sigma = \frac{1}{2} \left(2 + 2^*(s)\right) < 2^*(s)$.
Take a constant $C>0$ such that $\frac{1}{2^*(s)-\ep} |u|^{2^*(s)-\ep} + C \geq \frac{1}{2^*(s)} |u|^{\sigma}$ for all $0 < \ep < \sigma$ and $u \in \mathbb{R}$.
Then it follows that $c_{k,\ep} \leq A_k + C$ for $\ep \in (0,\sigma)$.

On the other hand, it is easily derived from the identity $\langle I_{\ep}' (u_{k,\ep}), u_{k,\ep}\rangle \geq 0$ that
\begin{equation}
\int_{\mathcal{C}}t^{1-2s} |\nabla U_{k,\epsilon}|^2 dx dt \leq C I_{\epsilon} (U_{k,\epsilon}) = C \cdot c_{k,\epsilon},
\end{equation}
where $C$ depends only on $N$ and $s$. Then, we have from the uniform boundedness of $c_{\ep, k}$ that
\begin{equation*}
\sup_{\ep > 0}\|U_{k,\ep}\| = \sup_{\ep >0} \int_{\mathcal{C}} t^{1-2s} |\nabla U_{k,\ep}|^2\, dx dt < \infty
\end{equation*}
and, consequently Theorem \ref{thm-uniform-bound} implies that there is a subsequence of $\{ U_{k,\ep_n}\}_{n \geq 1}$ such that $U_{k,\ep_n}$ converges strongly to a function $U_k$ in $H^1_0(t^{1-2s},\mathcal{C})$.
It then easily follows that $U_k$ solves the problem \eqref{u0inc} and satisfies $I (U_k) = c_k = \lim_{n\to \infty}c_{k,\ep_n}$ up to a subsequence.
Moreover, a standard argument (see e.g. \cite{CSS}) applies to show that
either $\{c_k\}_{k \in \mathbb{N}}$ has infinite number of elements  or
there is $m \in \mathbb{N}$ such that $c_k = c$ for all $k \geq m$ and infinitely many critical points correspond to the energy level $c$.
Therefore the problem \eqref{eq-main} is proved to have infinitely many solutions. This completes the proof of Theorem \ref{thm-main-1}.
\end{proof}

\appendix

\section{Proof of Lemma \ref{lem-average}}
This section is devoted to prove Lemma \ref{lem-average}. As a preliminary step, we first prove the following result.
\begin{lem}\label{lem-chain}
For $f \geq 0$ we suppose that $U \in H_{0}^{1}(t^{1-2s},\mathcal{C}) \cap L^{\infty}(\mathcal{C})$ is a weak solution of
\begin{equation}\label{eq-u-f}
\left\{ \begin{array}{ll} \textrm{div} (t^{1-2s} \nabla U) = 0 &\quad \textrm{in}~\mathcal{C},
\\
\partial_{\nu}^s U (x,0) = f(x)&\quad \textrm{on}~\Omega \times \{0\}.
\end{array}
\right.
\end{equation}
For $\gamma \in (1,\frac{2N+2}{2N+1})$ there exists a constant $C=C(N,\gamma)$ such that, for any $y \in \Omega$, $d >0$ and $0< r < \frac{1}{2} \textrm{dist}(y,\partial \Omega)$ we have
\begin{equation*}
\begin{split}
&\left( d^{-\gamma}  \fint_{B^{N+1}(y,r)} t^{1-2s} (U-a)_+^{\gamma} (x,t) dx dt \right)^{(2-\gamma)/\gamma}
\\
&\quad\quad \leq C d^{-\gamma}  \fint_{B^{N+1} ((y,0),2r)} t^{1-2s} (U -a)_{+}^{\gamma}(x,t) dx dt + C d^{-1} r^{-N+2s} \int_{B^{N}(y,r)} f (x) dx
\end{split}
\end{equation*}
provided that
\begin{equation}\label{eq-lem-cond}
m_s (\{ (x,t) \in B^{N+1}((y,0),2r) : a< U(x,t) < d \} ) \leq \frac{d^{-\gamma}}{2} \int_{B^{N+1}(x,r)} t^{1-2s} (U -a)_{+}^{\gamma} dx dt.
\end{equation}
Here the constant $C$ is independent of $a$, $d$ and $r$.
\end{lem}
\begin{proof}
In the proof, the notation $C$ denotes a generic constant independent of $a$, $d$ and $r$ that may change line by line. Without loss of generality, we may assume that $a=0$. By assumption \eqref{eq-lem-cond} we have
\begin{equation*}
\begin{split}
\int_{\{ z \in B_x^{N+1} (r) : U_{+} (z) < d\}} t^{1-2s} U_{+}^{\gamma} (z) dz &\leq d^{\gamma} m_s\left( \{ x \in B^{N+1}(x,r): 0<U < d\}\right)
\\
&\leq \frac{1}{2} \int_{B^{N+1}(z,r)} t^{1-2s} U_{+}^{\gamma}(z)dz,
\end{split}
\end{equation*}
where $m_s$ is the weighted volume defined in \eqref{eq-wm}. It gives
\begin{equation*}
\int_{\{ z \in B^{N+1}(x,r): 0<U (z) <d \}} t^{1-2s} U_{+}^{\gamma}(z) dz \leq  2 \int_{\{ z \in B^{N+1}(x,r): U >d\}} t^{1-2s} U_{+}^{\gamma}(z) dz.
\end{equation*}
Set $q=\frac{2\gamma}{2-\gamma}$ and
\begin{equation*}
w = \left( 1+ \frac{U_+}{d}\right)^{\gamma/q}-1 .
\end{equation*}
We can find a constant $C>0$ such that $\left(\frac{U_+}{d}\right)^{\gamma} \leq C w^{q}$ when $\frac{U_+}{d} \geq 1$. Using this we have
\begin{equation}\label{eq-uf-2}
\int_{\{ z \in B^{N+1}(x,r): U >d\}} t^{1-2s} U_{+}^{\gamma}(z) dz
\leq C d^{\gamma} \int_{B^{N+1}(x,r)} t^{1-2s} w^q (z) dz.
\end{equation}
Let $\eta \in C^{\infty}(\mathbb{R}^{N+1})$ be a cut-off function supported on $B^{N+1}({x},2r)$ such that $ \eta (z) = 1$ on $B^{N+1}(x,r)$ and $|\nabla \eta (z)|\leq C/r$. As $\gamma < \frac{2N+2}{2N+1}$, it holds that $q < \frac{2(N+1)}{N}$. Hence we may apply the weighted Sobolev inequality \eqref{eq-sobolev-weight} to get
\begin{equation}\label{eq-five}
\begin{split}
&\left( r^{-(N+2-2s)} \int_{B^{N+1}(x,r)} t^{1-2s} w^{q} dz \right)^{2/q}
\\
& \leq \left( r^{-(N+2-2s)} \int_{B^{N+1}(x,2r)} t^{1-2s} (\eta w)^q dz \right)^{2/q}
\\
&\leq r^{-(N+2-2s)} r^2 \int_{B^{N+1}(x,2r)} t^{1-2s}|\nabla (\eta w)|^2 dz \leq 2 r^{-(N-2s)} \int_{B^{N+1}(x,2r)}t^{1-2s} (|\nabla w\cdot \eta|^2 + |w \nabla \eta|^2 ) dz.
\end{split}
\end{equation}
We calculate
\begin{equation*}
\nabla w = \frac{\gamma}{qd} \left( 1+ \frac{U_+}{d}\right)^{\gamma/q -1} \nabla U_{+}.
\end{equation*}
In order to get a bound of $\int t^{1-2s} |\nabla w \cdot \eta|^2 dz$
 we take $V:= \left(1-\left( 1+\frac{U_{+}}{d} \right)^{2 \frac{\gamma}{q} -1} \right) \eta^2$ as a test function. Multiplying \eqref{eq-u-f} by $V$ and using Young's inequality we get
\begin{equation}\label{eq-uf-1}
\begin{split}
\int_{B^{N+1}(x,2r)} t^{1-2s} \nabla U \cdot \nabla V dz = \int_{B^{N} (x,2r)} \partial_{\nu}^s U (y,0) V(y,0) dy = C_s \int_{B^N (x,2r)} f(y) V(y,0) \eta^2 (y) dy.
\end{split}
\end{equation}
Note that
\begin{equation}
\nabla V = \left[ - \left( \frac{2\gamma}{q}-1\right) \left( 1 + \frac{U_{+}}{d}\right)^{\frac{2\gamma}{q} -2} \frac{\nabla U_{+}}{d} \right] \eta^2 + 2 \left( 1 - \left( 1 + \frac{U_+}{q}\right)^{\frac{2\gamma}{q} -1}\right) \eta \nabla \eta,
\end{equation}
 and we have $\frac{2\gamma}{q}-1 = -\gamma+1$. Injecting these equalities into  \eqref{eq-uf-1} we have
 \begin{equation}
\begin{split}
&\frac{(1-\gamma)}{d} \int_{B^{N+1}(x,2r)} t^{1-2s} |\nabla U_{+}|^2 \left( 1 + \frac{U_{+}}{d}\right)^{-\gamma} \eta^2 dy dt
\\
& = 2 \int_{B^{N+1}(x,2r)} t^{1-2s} \nabla U_{+}\left( 1 - \left( 1+\frac{U_{+}}{q}\right)^{-\gamma +1}\right) \eta \nabla \eta dy dt - C_s \int_{B^{N}(x,2r)} f (y) V(y,0) \eta^2 (y) dy
\end{split}
\end{equation}
As $\left( 1 - \left( 1+\frac{U_{+}}{q}\right)^{-\gamma +1}\right) \eta \leq 1$ we dedcue from the above identity that
 \begin{equation*}
\begin{split}
&\int_{B^{N+1}(x,2r)} t^{1-2s} |\nabla U_{+}|^2 \left( 1 + \frac{U_{+}}{d}\right)^{-\gamma} \eta^2 dy dt
\\
&\leq {Cd}\int_{B^{N+1}(x,2r)} t^{1-2s} |\nabla U_{+}| |\nabla \eta| dy dt  + {C d}\int_{B^{N}(x,2r)} f (y) V(y,0) \eta^2 (y) dy
\\
&\leq \frac{1}{2} \int_{B^{N+1}(x,2r)} t^{1-2s} |\nabla U_{+}|^2 \left( 1 + \frac{U_{+}}{d}\right)^{-\gamma} dy dt + C d^2 \int_{B^{N+1}(x,2r)} t^{1-2s} \left( 1 + \frac{U_+}{d}\right)^{\gamma} |\nabla \eta|^2 dz
\\
&\qquad + C d \int_{B^N (x,2r)} f (y) V(y,0) dy,
\end{split}
\end{equation*}
where we used Young's inequality in the second inequality. We can write this as
\begin{equation}\label{eq-chian}
\begin{split}
\int_{B^{N+1}(x,2r)} t^{1-2s} |\nabla U_{+}|^2 \left( 1+\frac{U_{+}}{d}\right)^{-\gamma}\eta^2 dz \leq & ~C d^2 \int_{B^{N+1}(x,2r)} t^{1-2s} \left( 1+ \frac{U_{+}}{d}\right)^{\gamma} |\nabla \eta|^2 dz
\\
&\qquad + C d \int_{B^{N}(x,2r)} f (y) V(y,0) dy.
\end{split}
\end{equation}
To estimate the first term in the right hand side, applying $|\nabla \eta| \leq C/r$ and condition \eqref{eq-lem-cond} once more, we deduce
\begin{equation*}
\begin{split}
&\int_{B^{N+1}(x,2r)} t^{1-2s} \left( 1+ \frac{U_{+}}{d}\right)^{\gamma} |\nabla \eta|^2 dz
\\
&~\quad \leq ~\frac{C}{r^2}\int_{B^{N+1}(x,2r)} t^{1-2s} \left(1 + \frac{U_{+}}{d}\right)^{\gamma} dx  \leq ~ \frac{C ~d^{-\gamma}}{r^2} \int_{B^{N+1}(x,2r) \cap \{ U>0\}} t^{1-2s} U^{\gamma} dz.
\end{split}
\end{equation*}
Plugging this into \eqref{eq-chian} we have
\begin{equation}\label{eq-uf-3}
\begin{split}
&\int_{B^{N+1}(x,2r)} t^{1-2s} |\nabla U_{+}|^2 \left( 1+\frac{U_{+}}{d}\right)^{-\gamma}\eta^2 dz 
\\
&\quad\leq\frac{C ~d^{2-\gamma}}{r^2} \int_{B^{N+1}(x,2r) \cap \{ U>0\}} t^{1-2s} U^{\gamma} dz + C d \int_{B^{N}(x,2r)} f (y) V(y,0) dy.
\end{split}
\end{equation}
On the other hand, we deduce from \eqref{eq-uf-2} and \eqref{eq-five} that
\begin{equation}
\begin{split}
&\left( d^{-\gamma} \int_{\{z \in B^{N+1}(x,r): U >d\}} t^{1-2s} U_{+}^{\gamma}(z) dz \right)^{2/q} 
\\
&\leq 2 r^{-(N-2s)} \int_{B^{N+1}(x,2r)} t^{1-2s} ( |\nabla w \cdot \eta|^2 + |w \nabla \eta|^2 ) dx
\\
&\leq \frac{C r^{-(N-2s)}}{d^2} \int_{B^{N+1}(x,2r)} t^{1-2s} \left( 1 + \frac{U_{+}}{d}\right)^{-\gamma} |\nabla U_{+}|^2 dz + 2r^{-(N-2s)} \int_{B^{N+1}(x,2r)} t^{1-2s} w^2 |\nabla \eta|^2 dz.
\end{split}
\end{equation}
Injecting \eqref{eq-uf-3} into the above inequality, we get
\begin{equation}\label{eq-final}
\begin{split}
&\left( d^{-\gamma}\int_{B^{N+1}(x,r)} t^{1-2s} U_{+}^{\gamma}(z) dz\right)^{2/q}
\\
&\leq \frac{C r^{-(N-2s)}}{d^2} \left[ r^{-2} d^{2-\gamma} \int_{B^{N+1}(x,r)}  t^{1-2s}U_{+}^{\gamma} dz + d \int_{B^n (x,r)} f (y) V(y,0) dy\right]
\\
&\qquad + r^{-(N-2s+2)} \int_{B^{N+1}(x,2r)} t^{1-2s} w^2 dz.
\end{split}
\end{equation}
The last term can be estimated by using H\"older's inequality and \eqref{eq-lem-cond} in the following way
\begin{equation*}
\begin{split}
\int_{B^{N+1}(x,2r)} t^{1-2s} w^2 dz &\leq \left( \int_{B^{N+1}(x,2r)} t^{1-2s} w^q dz \right)^{2/q} \left( m_s ( B(x,2r) \cap \{U >0\})\right)^{1-2/q}
\\
& \leq d^{-\gamma} \int_{B^{N+1}(x,2r) \cap \{U >0\}} t^{1-2s} U^{\gamma} dx.
\end{split}
\end{equation*}
Inserting this into \eqref{eq-final} we get the desired inequality. The proof is completed.
\end{proof}
\begin{proof}[Proof of Lemma \ref{lem-average}]
We denote $r_k = 2^{-k}$ for $k \in \mathbb{N}$. Take $\delta >0$ such that $\delta \leq \frac{2 m_s (B^{N+1}(x,r_{k}))}{m_s ( B^{N+1}(x, r_{k+1})) }$ whose value is independent of $k \in \mathbb{N}$. We set
\begin{equation*}
a_{k+1} = a_k + \left( \frac{1}{\delta} \fint_{B^{N+1}(x,r_{k+1})} t^{1-2s} (U-a_{k})_{+}^{\gamma} dx dt \right)^{1/\gamma}.
\end{equation*}
Let $d_k = a_{k+1} - a_k$. Then we have
\begin{equation*}
\begin{split}
\frac{1}{d_k^{\gamma}} \int_{B^{N+1} (x,r_{k+1})} t^{1-2s} (U -a_{k})_{+}^{\gamma} dx dt & = \delta~\cdot  m_s \big(B^{N+1} (x,r_{k+1})\big)
\\
& \geq 2 m_s \big(B^{N+1} (x,r_k)\big)
\\
& \geq 2 m_s \big( \{ (x,t) \in B^{N+1}(x,r_k): U(x,t) > a_k\}\big).
\end{split}
\end{equation*}
%Then it holds that
%\begin{equation*}
%m_s (\{ (x,t) \in B^{N+1}(x,r_k) : u(x,t) > a_k \} ) \leq d_k^{-\gamma} \int_{B^{N+1}(x,r_{k+1})} t^{1-2s} (u-a_k)_{+}^{\gamma} dx dt,
%\end{equation*}
%where $d_k = a_{k+1}- a_k$.
By  Lemma \eqref{lem-chain} we get
\begin{equation*}
\begin{split}
& \left( d_k^{-\gamma} r_k^{-(N+2-2s)} \int_{B^{N+1}(x,r_k)} t^{1-2s} (U-a_k)_{+}^{\gamma} (x,t) dx dt \right)^{2/q}
\\
& \qquad \leq C d_{k}^{-\gamma} r_k^{-(N+2-2s)} \int_{B^{N+1}(x,2r_k)} t^{1-2s} (U-a_{k})_{+}^{\gamma} (x,t) dxdt + C d_k^{-\gamma} r_k^{-(N-2s)} \int_{B^{N+1}(x,r_k)} f (y)dy
\\
&\qquad \leq C d_k^{-\gamma} r_k^{-(N+2-2s)} \int_{B^{N+1}(x,2r_k)} t^{1-2s} (U-a_{k-1})^{\gamma}(x,t) dxdt + C d_{k}^{-\gamma} r_k^{-(N-2s)} \int_{B^{N+1}(x,r_k)} f(y)dy
\\
&\qquad = C \delta \left[\frac{a_{k}-a_{k-1}}{a_{k+1} -a_{k}}\right]^{\gamma} + C d_k^{-1} r_k^{-(N-2s)} \int_{B^{N+1}(x,r_k)} f(y) dy.
\end{split}
\end{equation*}
Using the definition of $d_k$ we obtain
\begin{equation*}
\delta^{2/ q} \leq C \delta\left[\frac{a_k -a_{k-1}}{a_{k+1}-a_k}\right]^{\gamma} + C d_k^{-1} r_k^{-(N-2s)} \int_{B^{N+1}(x,r_k)} f(y) dy.
\end{equation*}
Note that $2/q =\frac{2-\gamma}{\gamma} <1$. We choose $\delta >0$ sufficiently small depending on $C$. Then it  follows that
\begin{equation*}
a_{k+1}-a_k \leq \frac{1}{2} (a_k - a_{k-1}) + C r_k^{-(N-2s)} \int_{B^{N+1}(x,r_k)} f(y) dy.
\end{equation*}
Summing up this, we have
\begin{equation*}
\begin{split}
a_k &\leq a_1 + C \sum_{j=1}^{k} r_j^{-(N-2s)} \int_{B^{N+1}(x,r_j) } f(y) dy
\\
& \leq a_1 +C \int^{1}_{r_k} \left( \frac{1}{w^{N-2s}} \int_{B^{N}(x,w)} f(y) dy\right) \frac{dw}{w}.
\end{split}
\end{equation*}
For given $r>0$ we take $k \in \mathbb{N}$ such that $r_{k+1} \leq r < r_k$. Then it follows from the above inequality that
\begin{equation*}
\left( \fint_{B^{N+1}(x,r)} t^{1-2s} U^{\gamma} ~dxdt\right)^{1/\gamma} \leq \fint_{B^{N+1}(x,1)} t^{1-2s} U^{\gamma}~ dxdt ~+ C \int^{1}_{r} \left( \frac{1}{w^{N-2s}} \int_{B^{N}(x,w)} f(y) dy\right) \frac{dw}{w}.
\end{equation*}
It completes the proof.
\end{proof}

\section{A Moser's iteration argument}
\begin{lem}\label{lem-harnack}
Let $r >1$ and consider a function $U \in D^{1}(t^{1-2s},\mathbb{R}^{N+1}_+)$ satisfying
\begin{equation}\label{eq-harnack}
\left\{ \begin{array}{ll}
\textrm{div}(t^{1-2s} \nabla U) = 0 &\quad \textrm{in}~B^{N+1}(0,5),
\\
\partial_{\nu}^s U = a (x) U &\quad \textrm{on}~B^N (0,5).
\end{array}
\right.
\end{equation}
Then, for each $q>1$, there exists a number $\epsilon = \epsilon (q)>$ such that, if $\| a\|_{L^{\frac{N}{2s}}A_0^N (\frac{1}{2},4)} \leq \epsilon$, then the following holds
\begin{equation*}
\| U\|_{L^q (A_0^{N+1}(1,2))} + \| U(\cdot,0)\|_{L^q (A_0^N (1,2))} \leq C \| U \|_{L^{r} (A_0^{N+1}(\frac{1}{2},4))},
\end{equation*}
where $C$ is a constant depending on $q$ and $\gamma$.
\end{lem}
\begin{proof}
We first take a smooth function $\phi \in C_c^{\infty}(B^{N+1} (0,5))$. Multiplying the function $|U|^{\beta-1}U \phi$ to \eqref{eq-harnack} we get
\begin{equation*}
\begin{split}
0 &= \int_{\mathbb{R}^{N+1}_{+}} \textrm{div}(t^{1-2s} \nabla U) |U|^{\beta-1}U \phi^2 dx dt
\\
&= - \int_{\mathbb{R}^{N+1}_{+}} t^{1-2s} \nabla U \nabla (|U|^{\beta-1}U\phi^2) dx dt + \int_{\mathbb{R}^{N}} (\partial_{\nu}^s U) |U|^{\beta-1}U \phi^2 (x, 0) dx.
\end{split}
\end{equation*}
A simple computation gives
\begin{equation}\label{eq-harnack-1}
\begin{split}
&\int_{\mathbb{R}^N} a(x) |U|^{\beta+1} \phi^2 (x,0) dx
\\
&= \frac{4\beta}{(1+\beta)^2} \int_{\mathbb{R}^{N+1}_{+}} t^{1-2s} |\nabla (U^{\frac{\beta+1}{2}})|^2 \phi^2 dx dt + \int_{\mathbb{R}^{N+1}} t^{1-2s} (\nabla U) |U|^{\beta} (2 \phi \nabla \phi )\, dx dt.
\end{split}
\end{equation}
Using Young's inequality we see
\begin{equation}\label{eq-harnack-2}
|(\nabla U) |U|^{\beta-1}U \phi\nabla \phi| = \frac{2}{\beta+1} |(\nabla |U|^{\frac{\beta+1}{2}} \phi) ( |U|^{\frac{\beta+1}{2}}\nabla \phi)| \leq \frac{1}{\beta+1}\left( |(\nabla |U|^{\frac{\beta+1}{2}}) \phi|^2 + ||U|^{\frac{\beta+1}{2}} \nabla \phi|^2 \right).
\end{equation}
We combine this inequality with  \eqref{eq-harnack-1} to deduce that
\begin{equation}\label{eq-harnack-3}
\begin{split}
&\int_{\mathbb{R}^N} a(x) |U|^{\beta+1} \phi^2 (x,0)dx + \frac{1}{\beta+1} \int_{\mathbb{R}^{N+1}_{+}} t^{1-2s} |U|^{\beta+1} |\nabla \phi|^2 dx dt
\\
&\quad\quad\quad\quad\quad\quad\quad\quad\quad\quad\quad\quad\geq \frac{3\beta}{(\beta+1)^2} \int_{\mathbb{R}^{N+1}_{+}} t^{1-2s} |\nabla (U^{\frac{\beta+1}{2}})|^2 \phi^2 ~ dx dt.
\end{split}
\end{equation}
Note that $(\nabla |U|^{\frac{\beta+1}{2}}) \phi = \nabla (|U|^{\frac{\beta+1}{2s}} \phi) - |U|^{\frac{\beta+1}{2}} \nabla \phi $. Then, using an elementary inequality $(a-b)^2  \geq \frac{a^2}{2} - 7b^2$ we deduce from \eqref{eq-harnack-3} that
\begin{equation}\label{eq-harnack-4}
\begin{split}
\int_{\mathbb{R}^N} a(x) |U|^{\beta+1} \phi^2 (x,0) ~dx +  & \frac{30 \beta}{(1+\beta)^2} \int_{\mathbb{R}^{N+1}_{+}} t^{1-2s} (|U|^{\frac{\beta+1}{2}} \nabla \phi)^2 ~ dx dt
\\
&\qquad \qquad \quad \geq \frac{2\beta}{(1+\beta)^2} \int_{\mathbb{R}^{N+1}_{+}} t^{1-2s} (\nabla (|U|^{\frac{\beta+1}{2}}\phi))^2 ~dx dt.
\end{split}
\end{equation}
The left-hand side can be estimated using H\"older's inequality and the Sobolev-trace inequality as follows.
\begin{equation*}
\begin{split}
\int_{\mathbb{R}^N} a(x)U^{\beta+1} \phi^2 (x,0) dx &\leq \| a\|_{\frac{N}{s}} \| U^{\frac{\beta+1}{2}} \phi (\cdot,0)\|_{\frac{2N}{N-2s}}^{2}
\\
& \leq C \epsilon   \int_{\mathbb{R}^{N+1}_{+}} t^{1-2s} |\nabla (U^{\frac{\beta+1}{2}} \phi)|^2 dx dt.
\end{split}
\end{equation*}
We assume that $\epsilon < \frac{1}{C \beta}$. Then it follows from the above inequality and \eqref{eq-harnack-4} that
\begin{equation*}
\frac{30 \beta}{(1+\beta)^2} \int_{\mathbb{R}^{N+1}_{+}} t^{1-2s} |U^{\frac{\beta+1}{2}} \nabla \phi|^2 dx dt \geq \frac{\beta}{(1+\beta)^2} \int_{\mathbb{R}^{N+1}_{+}} t^{1-2s}|\nabla (U^{\frac{\beta+1}{2}}\phi)|^2 dx dt.
\end{equation*}
Using the weighted Sobolev inequality and the Sobolev trace inequality we deduce that
\begin{equation}\label{eq-iterative}
\begin{split}
&\frac{30 \beta}{(1+\beta)^2} \int_{\mathbb{R}^{N+1}_{+}} t^{1-2s} |U^{\frac{\beta+1}{2}} \nabla \phi|^2 dx dt
\\
&\quad\quad \geq
\frac{C\beta}{(1+\beta)^2} \left[ \left( \int_{\textrm{supp}~\phi} t^{1-2s} |U|^{(\beta+1)\gamma} dx dt \right)^{\frac{2}{\gamma}} + \left( \int_{\textrm{supp}~\phi}  |U|^{\frac{2N}{N-2s}\cdot\frac{\beta+1}{2}} (x,0)
dx\right)^{\frac{N-2s}{N}} \right],
\end{split}
\end{equation}
where $\gamma =\frac{2(N+1)}{N}$.
We use this estimate iteratively. For any given $q>1$, applying \eqref{eq-iterative} with a suitable choice of $\beta$ and $\phi$ at each step, and H\"older's inequality we can deduce that
\begin{equation}
\left\| U\right\|_{L^q (A_0^{N+1}(1,2))} + \left\| U(\cdot,0)\right\|_{L^q (A_0^N (1,2))} \leq C \left\| U\right\|_{L^{r}(A_0^{N+1}(\frac{1}{2},4))}.
\end{equation}
The proof is complete.
\end{proof}

\section{Local Pohozaev identity}
%\begin{lem}
%\begin{equation}\label{eq-identity-3}
%\begin{split}
%&\int_{\mathbb{R}^{N+1}_{+}} t^{1-2s} |\nabla U|^2 \phi dx dt
%\\
%&\quad\quad = \frac{1}{2} \int_{\mathbb{R}^{N+1}_{+}}  U^2 \textrm{div} (t^{1-2s} \nabla \phi) dx dt + \frac{1}{2}\int_{\mathbb{R}^n} U^2 (x,0) \partial_{\nu}^s \phi (x,0) dx -\int_{\mathbb{R}^n}U \partial_{\nu}^s U(x,0) \phi (x,0) dx.
%\end{split}
%\end{equation}
%\end{lem}
%\begin{proof}
% We have
%\begin{equation}\label{eq-identity-1}
%\begin{split}
%0&=\int_{\mathbb{R}^{N+1}_{+}} \textrm{div}(t^{1-2s} \nabla U)  U \phi dxdt
%\\
%&= \int_{\mathbb{R}^{N+1}_{+}}\left( t^{1-2s} |\nabla U|^2 \phi + t^{1-2s} (\nabla U U ) \nabla \phi \right) dx dt + \int_{\mathbb{R}^{N}} U \partial_{\nu}^s U (x,0) \phi (x,0) dx.
%\end{split}
%\end{equation}
%Note that
%\begin{equation}\label{eq-identity-2}
%\begin{split}
%\int_{\mathbb{R}^{N+1}_{+}} t^{1-2s}(\nabla U U) \nabla \phi dx dt &=\frac{1}{2} \int_{\mathbb{R}^{N+1}_{+}} t^{1-2s}  \nabla (U^2) \nabla \phi dx dt
%\\
%&= -\frac{1}{2}\int_{\mathbb{R}^{N+1}_{+}} U^2 \textrm{div} (t^{1-2s} \nabla \phi) dx dt - \frac{1}{2} \int_{\mathbb{R}^n} U^2 (x,0) \partial_{\nu}^{s} \phi (x,0) dx
%\end{split}
%\end{equation}
%Combining \eqref{eq-identity-1} and \eqref{eq-identity-2}gives the desired identity.
%\end{proof}
For $D \subset \mathbb{R}^{N+1}_{+}$ we define the following sets $\partial_{+} D = \{ (x,t) \in \mathbb{R}^{N+1}_{+} : (x,t) \in \partial D \quad \textrm{and}\quad t >0 \}$,
 and  $\partial_{b} D = \partial D \cap \mathbb{R}^N \times \{0\}.$
We state the following.
\begin{lem} Let $E \subset \mathbb{R}^{N+1}_{+}$ and we assume that a function $U$ is a solution of
\begin{equation}
\left\{\begin{array}{ll}
\textrm{div}(t^{1-2s} \nabla U) = 0 &\quad \textrm{in}~E,
\\
\partial_{\nu}^{s} U = f (U)&\quad \textrm{on}~\partial_b E.
\end{array}
\right.
\end{equation}
Then, for $D \subset E$ we have the following identity.
\begin{equation}\label{eq-local}
\begin{split}
&\ C_s \left\{N \int_{\partial_{b} D} F(U) dx - \left(\frac{N-2s}{2}\right) \int_{ \partial_{b} D} U f(U) dx\right\}\\
&= \int_{ \partial_{+} D} t^{1-2s}\left<(z-x_j , \nabla U) \nabla U - (z-x_j) \frac{|\nabla U|^2}{2}, \nu\right> dS
\\
& \quad + \left(\frac{N-2s}{2}\right) \int_{\partial_{+} D} t^{1-2s} U \frac{\partial U}{\partial \nu} dS +  \int_{\partial\partial_{b} D} (x,\nu) F(U) dS_x,
\end{split}
\end{equation}
where $F (s) = \int^{s}_0 f (t) dt$.
\end{lem}
\begin{proof}
We have the identity
\begin{equation}\label{eq-poho-root}\textrm{div} \biggl\{ t^{1-2s} (z,\nabla U ) \nabla U - t^{1-2s} \frac{|\nabla U|^2}{2} z \biggr\} + \biggl( \frac{N-2s}{2} \biggr) t^{1-2s}|\nabla U |^2=0.
\end{equation}
Integrating this over the domain $D$, we get
\begin{multline}\label{dxdy0}
 \int_{ \partial_{+} D} t^{1-2s}\left<(z, \nabla U) \nabla U - z \frac{|\nabla U|^2}{2}, \nu\right> dS
+ C_s \int_{\partial_{b} D}  (x, \nabla_x U) \partial_{\nu}^s U  dx
\\
\quad = - \left( \frac{N-2s}{2}\right) \int_{D} t^{1-2s}|\nabla U|^2 dx dt.
\end{multline}
By using $\partial_{\nu}^s U = f(U)$ and performing integration by parts, we deduce that
\begin{align*}
\int_{\partial_{b} D} (x, \nabla_x U) \partial_{\nu}^s U dx
&=\int_{\partial_{b} D} (x, \nabla_x U) f (U) dx
\\
&=\int_{\partial_{b} D} x \cdot \nabla_x F(U) dx
\\
&= -N \int_{\partial_{b} D} F(U) dx +\int_{\partial \partial_{b} D} (x,\nu) F(U) dS_x
\end{align*}
and
\[\int_{D_r} t^{1-2s} |\nabla U|^2 dx dt= C_s\int_{ \partial_{b} D} U f(U) dx + \int_{\partial_+ D} t^{1-2s} U \frac{\partial U}{\partial \nu} dS.\]
Then \eqref{dxdy0} gives the desired identity.
\end{proof}

\medskip
\noindent \textbf{Acknowledgments}

\medskip
The first author thanks his advisor Prof. Rapha\"el Ponge for his support and encouragement. This work is a part of the thesis of the first author. He was supported by the Global Ph.D Fellowship of the Government of South Korea 300-20130026.
Both authors thank Prof. Ki-Ahm Lee for his helpful discussions on this subject.

\end{document}